\RequirePackage{rotating}
\documentclass[smallextended,envcountsect,envcountsame]{svjour3}
\usepackage{algorithm}
\usepackage{graphicx}
\usepackage{mathtools}
\usepackage{mathrsfs}
\usepackage{fix-cm}
\usepackage{multirow}
\usepackage{amsmath}
\usepackage{amssymb}
\usepackage{latexsym}
\usepackage{dsfont}
\usepackage{xcolor}
\usepackage{cite}
\usepackage{dsfont}
\usepackage{enumitem}
\usepackage{todonotes}
\usepackage{hyperref}
\usepackage{booktabs}
\hypersetup{
colorlinks=true,        
linkcolor=blue,    
citecolor=darkgreen,         
urlcolor=blue}
\definecolor{darkgreen}{RGB}{0,150,0}
\definecolor{C0}{RGB}{31,119,180}
\definecolor{C1}{RGB}{255, 127, 14}
\usepackage{algpseudocode}

\newtheorem{experiment}{Experiment}
\let\oldexperiment\experiment
\renewcommand{\experiment}{\oldexperiment\normalfont}

\def\disp{\displaystyle}

\def\Limsup{\mathop{{\rm Lim}\,{\rm sup}}}

\def\tto{\;{\lower 1pt \hbox{$\rightarrow$}}\kern -10pt
\hbox{\raise 2pt \hbox{$\rightarrow$}}\;}

\def\Hat{\widehat}
\def\hat{\widehat}

\def\Bar{\overline}

\def\ve{\varepsilon}

\def\R{\mathbb{R}}
\def\N{\mathbb{N}}

\def\ox{\bar{x}}
\def\oy{\bar{y}}

\def\ov{\bar{v}}

\def\co{\mbox{\rm co}\,}
\def\gph{\operatorname{gph}}
\def\epi{\mbox{\rm epi}\,}

\def\dom{\mbox{\rm dom}\,}

\def\clm{\mbox{\rm clm}\,}

\def\tilde{\widetilde}
\def\ph{\varphi}
\def\emp{\emptyset}

\def\oR{\Bar{\R}}
\def\N{\mathbb{N}}
\def\lm{\lambda}

\def\tt{\tilde t}

\def\tr{T}

\newcounter{count}

\newcommand{\Rex}{\overline{\mathbb{R}}}
\let\epsilon\varepsilon
\DeclareMathAlphabet{\mathpzc}{OT1}{pzc}{m}{it}

\def\X{{\R^n}}
\def\w{{v}}
\def\Y{\R^{m}}

\usepackage{framed}
\usepackage{mdframed}
\usepackage{pgfplots}
\pgfplotsset{compat = newest}
\usepackage{caption}
\usepackage{subcaption}
\newcommand{\MoreauYosida}[2]{{\mathtt{e}}_{#2} #1}
\newcommand{\Prox}[2]{{\mathtt{Prox}}_{#2 #1}}
\newcommand{\Asp}[2]{{{\mathtt{A}}_{#1 }#2}}

\begin{document}

\titlerunning{Coderivative-Based Semi-Newton Method in Nonsmooth Difference Programming}
\title{Coderivative-Based Semi-Newton Method\\ in Nonsmooth Difference Programming
\thanks{Research of the first author was partially supported by the Ministry of Science, Innovation and Universities of Spain and the European Regional Development Fund (ERDF) of the European Commission, Grant PGC2018-097960-B-C22, and by the Generalitat Valenciana, grant AICO/2021/165. Research of the second author was partially supported by the USA National Science Foundation under grants DMS-1808978 and DMS-2204519, by the Australian Research Council under Discovery Project DP-190100555, and by the Project 111 of China under grant D21024. Research of the third author was partially supported by grants: Fondecyt Regular 1190110 and Fondecyt Regular 1200283.}}
\subtitle{}
\author{Francisco J. Arag\'{o}n-Artacho \and \mbox{Boris S. Mordukhovich} \and \mbox{Pedro P\'erez-Aros}}

\institute{Francisco J. Arag\'{o}n-Artacho  \at Department of Mathematics, University of Alicante, Alicante, Spain\\ \email{francisco.aragon@ua.es} \and
Boris S. Mordukhovich \at Department of Mathematics, Wayne State University, Detroit, Michigan 48202, USA\\ \email{boris@math.wayne.edu}
\and Pedro P\'erez-Aros \at Instituto de Ciencias de la Ingenier\'ia, Universidad de O'Higgins, Rancagua, Chile\\
\email{pedro.perez@uoh.com}} 

\date{\today}

\maketitle

\begin{abstract}
This paper addresses the study of a new class of nonsmooth optimization problems, where the objective is represented as a difference of two generally nonconvex functions. We propose and develop a novel Newton-type algorithm to solving such problems, which is based on the coderivative generated second-order 
subdifferential (generalized Hessian) and employs advanced tools of variational analysis. Well-posedness properties of the proposed algorithm are derived under fairly general requirements, while constructive convergence rates are established by using additional assumptions including the Kurdyka--{\L}ojasiewicz condition. We provide applications of the main algorithm to solving a general class of nonsmooth nonconvex problems of structured optimization that encompasses, in particular, optimization problems with explicit constraints. Finally, applications and numerical experiments are given for solving practical problems that arise in biochemical models, constrained quadratic programming, etc., where advantages of our algorithms are demonstrated in comparison with some known techniques and results.
\end{abstract}\vspace*{-0.05in}

\vspace*{-0.05in}

\vspace*{-0.2in}

\section{Introduction}\label{intro}\vspace*{-0.05in}

The primary mathematical model considered in this paper is described by
\begin{equation}\label{EQ01}
\min_{x\in \R^n} \varphi(x):=g(x)-h(x),
\end{equation}
where $g:\mathbb{R}^n \to \R$ is of {\em class $\mathcal{C}^{1,1}$} (i.e., the collection of $\mathcal{C}^1$-smooth functions with locally Lipschitzian derivatives), and where $h: \mathbb{R}^n \to \R$ is a locally Lipschitzian and {\em prox-regular} function; see below. Although \eqref{EQ01} is a problem of unconstrained optimization, it will be shown below that a large class of constrained optimization problems can be reduced to this form. In what follows, we label the optimization class in \eqref{EQ01} as problems of {\em difference programming}.  

The difference form \eqref{EQ01} reminds us of problems of {\em DC $($difference of convex$)$ programming}, which have been intensively studied in optimization with a variety of practical applications; see, e.g., \cite{Aragon2020,Artacho2019,AragonArtacho2018,Oliveira_2020,hiriart,Toh,Tao1997,Tao1998,Tao1986} and the references therein. However, we are not familiar with a systematic study of the class of difference programming problems considered in this paper.

Our main goal here is to develop an efficient numerical algorithm to solve the class of difference programs \eqref{EQ01} with subsequent applications to nonsmooth and nonconvex problems of particular structures, problems with geometric constraints, etc. Furthermore, the efficiency of the proposed algorithm and its modifications is demonstrated by solving some practical models for which we conduct numerical experiments and compare the obtained results with previously known developments and computations by using other algorithms.   
The proposed algorithm is of a {\em regularized damped Newton type} with a {\em novel choice of directions} in the iterative scheme providing a {\em global convergence} of iterates to a stationary point of the cost function. At the {\em first order}, the novelty of our algorithm, in comparison with, e.g., the most popular {\em DCA algorithm} by Tao et al. \cite{Tao1997,Tao1998,Tao1986} and its {\em boosted} developments by Arag\'on-Artacho et al.\cite{Aragon2020,Artacho2019,AragonArtacho2018,MR4078808}
in DC programming, is that instead of a convex subgradient of $h$ in \eqref{EQ01}, we now use a {\em limiting subgradient} of $-h$. No second-order information on $h$ is used in what follows. Concerning the other function $g$ in \eqref{EQ01}, which is nonsmooth of the {\em second-order}, our algorithm replaces the classical Hessian matrix by the {\em generalized Hessian/second-order subdifferential} of $g$ in the sense of Mordukhovich \cite{m92}. The latter construction, which is defined as the coderivative of the limiting subdifferential has been well recognized in variational analysis and optimization due its comprehensive calculus and explicit evaluations for broad classes of extended-real-valued functions arising in applications. We refer the reader to, e.g., \cite{chhm,dsy,Helmut,hmn,hos,hr,2020arXiv200910551D,MR3823783,MR2191744,mr,os,yy} and the bibliographies therein for more details. Note also that the aforementioned generalized Hessian has already been used in differently designed algorithms of the Newton type to solve
optimization-related problems of different nonsmooth structures in comparison with \eqref{EQ01}; see \cite{Helmut,2020arXiv200910551D,jogo,2021arXiv210902093D,BorisEbrahim}. Having in mind the discussions above, we label the main algorithm developed in this paper as the {\em regularized coderivative-based damped semi-Newton method} (abbr.\ RCSN).   

The rest of the paper is organized as follows. Section~\ref{sec:2} recalls constructions and statements from variational analysis and generalized differentiation, which are broadly used in the formulations and proofs of the major results. Besides well-known facts, we present here some new notions and further elaborations. 

In Section~\ref{sec:3}, we design our {\em main RCSN algorithm}, discuss each of its steps, and establish various results on its performance depending on imposed assumptions whose role and importance are illustrated by examples. Furthermore, Section~\ref{sec:4} employs the {\em Kurdyka-{\L}ojasiewicz {\rm(KL)} property} of the cost function to establish quantitative convergence rates of the RCSN algorithm depending on the exponent in the KL inequality. 

Section~\ref{sec:5} addresses the class of (nonconvex) problems of {\em structured optimization} with the cost functions given in the form $f(x)+\psi(x)$, where $f\colon\R^n\to\R$ is a twice continuously differentiable function with a Lipschitzian Hessian (i.e., of class ${\cal C}^{2,1}$), while $\psi\colon\R^n\to\oR:=(-\infty,\infty]$ is an extended-real-valued prox-bounded function. By using the {\em forward-backward envelope} \cite{MR3845278} and the associated {\em Asplund function} \cite{asplund}, we reduce this class of structured optimization problems to the difference form \eqref{EQ01} and then employ the machinery of RCSN to solving problems of this type. As a particular case of RCSN, we design and justify here a new {\em projected-like Newton algorithm} to solve optimization problems with geometric constraints given by general closed sets.

Section~6 is devoted to implementations of the designed algorithms and {\em numerical experiments} in two different problems arising in practical modeling. Although these problems can be treated after some transformations by DCA-like algorithms, we demonstrate in this section numerical advantages of the newly designed algorithms over the known developments in both smooth and nonsmooth settings. The concluding Section~\ref{sec:7} summarizes the major achievements of the paper and discusses some directions of our future research.\vspace*{-0.4in}

\section{Tools of Variational Analysis and Generalized Differentiation}\label{sec:2}\vspace{-0.1in}

Throughout the entire paper, we deal with finite-dimensional Euclidean spaces and use the standard notation and terminology of variational analysis and generalized differentiation; see, e.g., \cite{MR3823783,MR1491362}, where the reader can find the majority of the results presented in this section. Recall that 
$\mathbb{B}_r(x)$ stands for the closed ball centered at $x\in\mathbb{R}^n$ with radius $r> 0$ and that $\N:=\{1,2,\ldots\}$.

Given a set-valued mapping  $F: \R^n \tto \R^m$, its {\em graph}  is the set $	\gph F:=\big\{ (v,w) \in \X\times \Y\;|\; w\in F(x)\big\}$, while the (Painlev\'e--Kuratowski) {\em outer limit} of $F$ at $x\in\R^n$ is defined by
\begin{equation}\label{pk}
\Limsup_{u\to x}F(u):=\big\{y\in\R^m\;\big|\;\exists\,u_k\to x,\,y_k\to y,\;y_k\in F(u_k)\;\mbox{as}\;k\in\N\big\}.
\end{equation}

 For a nonempty set $C\subseteq\mathbb{R}^n$, the (Fr\'echet) {\em regular normal cone}  and (Mordukhovich) {\em basic/limiting  normal cone} at $x\in C$ are defined, respectively, by
\begin{equation}\label{rnc}
\begin{aligned}
\Hat{N}(x;C)=\Hat N_C(x):&=\Big\{ x^*\in \mathbb{R}^n\;\Big|\;\limsup\limits_{ u \overset{C}{\to} x }  \Big\langle x^\ast,\frac{u - x}{\| u - x  \| }\Big\rangle\leq 0\Big\},\\
N(x;C)=N_C(x):&=\Limsup\limits_{u \overset{C}{\to} x}\Hat{N}(u;C),
\end{aligned}
\end{equation}
where  ``$u\overset{C}{\to} x$'' means that $u \to x$ with $u \in C$. We use the convention $\Hat N(x;C)=N(x;C):=\emp$ if $x\notin C$. The {\em indicator function} $\delta_C(x)$ of $C$ is equal to 0 if $x\in C$ and to $\infty $ otherwise.

For a lower semicontinuous (l.s.c.) function  $f:\X\to\Rex$, its {\em domain}  and {\em epigraph} are given by $\dom f := \{ x\in \R^n \mid f(x) < \infty \}$ and $\epi f:=\{ (x,\alpha) \in \X\times \R \;|\;f(x)\leq \alpha\},$  respectively. The {\em regular} and {\em basic subdifferentials} of $f$ at $x\in\dom f$ are defined by
\begin{equation}\label{sub}
\begin{aligned}
\Hat{\partial} f(x)&:=\big\{ x^\ast\in \mathbb{R}^n\mid(x^\ast ,-1) \in \Hat{N}\big((x,f(x));\epi f\big)\big\},\\
\partial f(x) &:= \big\{ x^\ast\in \mathbb{R}^n\mid(x^\ast ,-1) \in N\big((x,f(x));\epi f\big)\big\},
\end{aligned}	
\end{equation}
via the corresponding normal cones \eqref{rnc} to the epigraph. The function $f$ is said to be {\em lower/subdifferentially regular} at $\bar{x}\in\dom f$ if $\partial  f(\bar{x}) =\Hat{\partial}f(\bar{x})$.
 
Given further a set-valued mapping/multifunction $F: \X \tto \Y$,
the {\em regular} and {\em basic coderivatives} of $F$ at $(x,y)\in\gph F$ are defined  for all $y^*\in\Y$ via the corresponding normal cones \eqref{rnc}  to the graph of $F$, i.e.,
\begin{equation}\label{cod}
\begin{aligned}
\Hat{D}^\ast F(x,y) (y^\ast)&:=\big\{  x^\ast \in \X\;\big|\;(x^\ast,-y^\ast) \in  \Hat{N}\big( (x,y); \gph F\big)\big\},\\
{D}^\ast F(x,y) (y^\ast)&:=\big\{x^\ast \in \X\;\big|\;(x^\ast,-y^\ast) \in  N\big( (x,y); \gph F\big)\big\},
\end{aligned}
\end{equation}
where $y$ is omitted if $F$ is single-valued at $x$. When $F$ is single-valued and locally Lipschitzian around $x$, the basic coderivative has the following representation via the basic subdifferential of the scalarization
\begin{align}\label{coder:sub}
{D}^\ast F(x) (y^\ast) = \partial \langle y^\ast, F\rangle (x), \text{ where }   \langle y^\ast, F\rangle (x):= \langle y^\ast, F(x)\rangle.
\end{align}

Recall that a set-valued mapping $F: \mathbb{R}^n \tto \mathbb{R}^m $ is {\em strongly metrically subregular} at $(\bar{x},\bar{y})\in \gph F$ if there exist $\kappa,\epsilon>0$ such that
\begin{align}\label{def:stron_subreg}
\| x -\bar{x}\| \leq \kappa \| y -  \bar{y}\|\; \text{ for all }\;(x,y) \in \mathbb{B}_\epsilon(\bar{x},\bar{y}) \cap\gph F.
\end{align}
It is well-known that this property of $F$ is equivalent to the {\em calmness} property of the inverse mapping $F^{-1}$ at $(\oy,\ox)$. In what follows, we use the calmness property of single-valued mappings $h\colon\R^n\to\R^m$ at $\ox$ meaning that there exist positive numbers $\kappa$ and $\ve>0$ such that
\begin{equation}\label{calm}
\|h(x)-h(\bar{x})\|\leq\kappa\|x-\bar{x}\|\;\text{ for all }\;x\in\mathbb{B}_\epsilon(\bar{x}).
\end{equation}
The infimum of all $\kappa>0$ in \eqref{calm}  is called the \emph{exact calmness bound} of $h$ at $\bar{x}$ and is denoted it by $\clm h(\bar{x})$. On the other hand, a multifunction $F\colon\R^n\tto\R^m$ is {\em strongly metrically regular} around $(\bar{x},\bar{y}) \in  \gph F$ if its inverse $F^{-1}$  admits a single-valued and Lipschitz continuous localization around this point.

Along with the (first-order) basic subdifferential in \eqref{sub}, we consider the {\em second-order subdifferential/generalized Hessian} of $f: \X \to \Rex$ at $x\in \dom f$ relative to $x^\ast\in\partial f(x)$ defined by
\begin{equation}\label{2nd}
\partial^2f(x,x^\ast) (v^\ast) = \left(D^\ast\partial f\right) (x,x^\ast)(v^\ast),\quad v^\ast\in \X
\end{equation}
and denoted by $\partial^2f(x) (v^\ast)$ when $\partial f(x)$ is a singleton. If $f$ is twice continuously differentiable ($\mathcal{C}^2$-smooth) around $x$, then $\partial^2 f(x)(v^\ast)=\{\nabla^2 f(x)v^\ast\}$.

Next we introduce an extension of the notion of positive-definiteness for multifunctions, where the the corresponding constant may not be positive.\vspace*{-0.1in}

\begin{definition}\label{def:lower-def}
Let $F: \mathbb{R}^n \tto \mathbb{R}^n$ and $\xi\in\mathbb{R}$. Then $F$ is \emph{$\xi$-lower-definite} if
\begin{align}\label{Stron-semide}
\langle y, x\rangle \geq \xi\|x\|^2\;\text{ for all }\;(x,y)\in\gph F. 
\end{align}
\end{definition}

\begin{remark}\label{rem:definite}
We can easily check the following:

(i) For any symmetric matrix $Q$ with the smallest eigenvalue $\lambda_{\min}(Q)$, the function $f(x)=Qx$ is $\lambda_{\min}(Q)$-lower-definite.

(ii) If a function $f: \R^n \to  \Rex$ is {\em strongly convex} with modulus $\rho>0$, (i.e., $ f -\frac{\rho }{2}\| \cdot \|^2$ is convex), it follows from \cite[Corollary~5.9]{MR3823783} that $\partial^2 f(x,x^\ast)$ is $\rho$-lower-definite for all $(x,x^\ast)\in\gph\partial f$.

(iii) If $F_1,F_2: \mathbb{R}^n \tto \mathbb{R}^n$ are $\xi_1$ and $\xi_2$-lower-definite, then the sum $F_1+F_2$ is $(\xi_1+\xi_2)$-lower-definite.
\end{remark}\vspace*{-0.05in}

Recall next that a function $f: \R^n \to \Rex$ is \emph{prox-regular} at $\bar{x}\in\R^n$ \emph{for} $\bar{v} \in \partial f(\bar{x})$ if it is l.s.c.\  around $\bar{x}$ and there exist $\epsilon>0$ and $r\geq 0$ such that
\begin{align}\label{proregularity}
f(x') \geq f(x) + \langle v,  x'-x\rangle  - \frac{r}{2}\| x'- x\|^2
\end{align}
whenever $x,x'\in\mathbb{B}_\epsilon(\bar{x})$ with $ f(x) \leq  f(\bar{x}) + \epsilon$ and $v\in \partial f(x) \cap \mathbb{B}_\epsilon(\bar{v})$. If this holds for all $\bar{v} \in \partial f(\bar{x})$, $f$ is said to be \emph{prox-regular at} $\bar x$.\vspace*{-0.05in}

\begin{remark}
The class of prox-regular functions has been well-recognized in modern variational analysis. It is worth mentioning that if $f$ is a locally Lipschitzian function around $\bar{x}$, then the following properties of $f$ are equivalent:  (i)  prox-regularity at $\bar{x}$, (ii) lower-$\mathcal{C}^2$ at $\bar{x}$, and (iii) primal-lower-nice at $\bar{x}$; see, e.g., \cite[Corollary 3.12]{MR2101873} for more details.
\end{remark}\vspace*{-0.05in}

Given a function $f: \mathbb{R}^n \to \Rex$ and $\bar{x}\in\dom f$, the {\em upper directional derivative} of $f$ at $\bar{x}$ with respect to $d\in\R^n$  is defined by
\begin{equation}\label{UDD}
f'(\bar{x};d):=\limsup\limits_{t\to 0^+} \frac{f(\bar{x}+td) - f(\bar{x})}{t}.
\end{equation}
The following proposition establishes various properties of prox-regular functions used below. We denote the convex hull of a set by ``co".

\begin{proposition}\label{Lemma:Dire01}
Let $f: \mathbb{R}^n \to \Rex$ be locally Lipschitzian around $\ox$ and prox-regular at this point. Then $f$ is lower regular at $\bar{x}$, $\co\partial (-f)(\bar{x}) = -\partial f(\bar{x})$, and for any $d\in\R^n$ we have the representations
\begin{align} \label{Dire01}
(-f)'(\bar x;d)=\inf\big\{\langle w, d\rangle\;\big|\;w\in \partial (-f)(\bar x)\big\}= \inf\big\{\langle  w, d\rangle\;\big|\;w\in-\partial f(\bar{x})\big\}.	
\end{align}
\end{proposition}
\begin{proof}
First we fix an arbitrary subgradient $\bar{v} \in \partial f(\bar{x})$ 
 and deduce from \eqref{proregularity} applied to $x=\bar{x}$ and $v=\bar{v}$ that
\begin{equation*}
f(x') \geq f(\bar{x}) +\langle  \bar{v}, x'-\bar{x}\rangle -\frac{r}{2} \| x' - \bar{x}\|^2\;\text{ for all }\;x'\in\mathbb{B}_\epsilon(\bar{x}).
\end{equation*}
Passing to the limit as $x' \to \bar{x}$ tells us that
\begin{align*}
\liminf_{x'\to \bar{x}} \frac{ f(x') - f(\bar{x}) - \langle  \bar{v}, x'-\bar{x}\rangle  }{   \| x' - \bar{x}\|  } \geq 0,
\end{align*} 
which means that $\ov\in\Hat\partial f(\ox)$ and thus shows that 
$f$ is lower regular at $\bar{x}$. By the Lipschitz continuity of $f$ around $\bar{x}$ and the convexity of the set $\Hat{\partial} f(\bar{x})$, we have that $\Hat{\partial} f(\bar{x})=\partial f(\bar{x}) =\co \partial f(\bar{x}) = \overline{\partial } f(\bar{x})$,  where $\overline{\partial }$ denotes the (Clarke) {\em generalized gradient}. It follows from $ \overline{\partial } (-f)(\bar{x}) =-\overline{\partial} f(\bar{x})$ that $ \overline{\partial } (-f)(\bar{x}) =-\partial f(\bar{x})$, which implies therefore that $\partial (-f)(\bar{x}) \subseteq -\partial f(\bar{x})$. 	

Pick $v\in  \partial (-f)(\bar{x})$, $d\in \mathbb{R}^n$ and find by the prox-regularity of $f$ at $\bar{x}$ for $-v\in\partial f(\bar{x})$ that there exists $r>0$ such that
\begin{align*}
\langle v,d \rangle+\frac{rt}{2}\| d\|^2 \geq \frac{-f(\bar x+td) + f(\bar x)}{t}
\end{align*}
if $t>0$ is small enough. This yields $(-f)'(\bar x;d) \leq 	\langle v,  d \rangle $ for all $\w\in \partial (-f)(\bar x) $ and thus verifies the inequality ``$\leq $''  in  the first representation of \eqref{Dire01}.

To prove the opposite inequality therein, take $t_k \to 0^+$ such that
\begin{align*}
\lim_{k\to \infty} \frac{ -f(\bar{x} + t_k d ) + f(\bar{x}) }{ t_k}= (-f)'(\bar x;d).
\end{align*} 
Employing the mean value theorem from \cite[Corollary~4.12]{MR3823783}) gives us
\begin{align*}
f(\bar{x} + t_k d ) - f(\bar{x}) = t_k \langle v_k,d\rangle \text{ for some } v_k \in \partial f(\bar{x} + \lambda_k t_k d) \text{ with } \lambda_k   \in ( 0,1).
\end{align*}
It follows from the Lipschitz continuity of $f$ that $\{v_k\}$ is bounded,
and so we can assume that $v_k \to \bar{v}\in  \partial f(\bar{x})$. Therefore,
\begin{equation*}
\begin{array}{ll}
(-f)'(\bar x;d)=\langle -\bar{v}, d\rangle \geq \inf\big\{\langle  w, d\rangle\;\big|\; w \in -\partial f(\bar{x})\big\} \\
=\inf\big\{\langle  w, d\rangle\;\big|\;w \in \co \partial (-f)(\bar{x})        \big\}=\inf\big\{\langle  w, d\rangle\;\big|\;w \in  \partial (-f)(\bar{x})      \big\},
\end{array}
\end{equation*}
which verifies \eqref{Dire01} and completes the proof of the proposition.
\end{proof}\vspace*{-0.05in}

Next we define the notion of stationary points for problem \eqref{EQ01} the finding of which is the goal of our algorithms.

\begin{definition}\label{def:stationary}
Let $\varphi=g-h$ be the cost function in \eqref{EQ01}, where $g$ is of class $\mathcal{C}^{1,1}$ around some point $\ox$, and where $h$ is locally Lipschitzian around $\ox$ and prox-regular at this point. Then $\bar{x}$ is a \emph{stationary point} of \eqref{EQ01} if $0\in\partial\varphi(\bar{x})$.
 \end{definition}\vspace*{-0.2in}
 
 \begin{remark} 
 The stationarity notion $0 \in \partial\varphi(\bar{x})$, expressed via 
 the limiting subdiffential, is known as the {\em M$($ordukhovich$)$-stationarity}. Since no other stationary points are considered in this paper, we skip ``M" in what follows. Observe from the subdifferential sum rule in our setting that $\bar{x}$ is a stationary point in \eqref{EQ01} if and only if $0\in\nabla g(\bar{x}) +\partial(-h)(\bar{x})$. Thus every stationary point $\bar{x}$ is a {\em critical point} in the sense that $0\in\nabla g(\bar{x}) - \partial h(\bar{x})$. By Proposition~\ref{Lemma:Dire01}, the latter can be equivalently described in terms of the generalized gradient and also via the {\em symmetric subdifferential} \cite{MR3823783} of $\ph$ at $\ox$ defined by
\begin{equation}\label{sym}
\partial^0\ph(\ox):=\partial\varphi(\bar{x})\cup\big(-\partial(-\varphi)(\bar{x})\big)
\end{equation}
which possesses the plus-minus symmetry $\partial^0(-\ph(\ox))=-\partial^0(\ph(\ox))$. When both $g$ and $h$ are convex, the classical DC algorithm~\cite{Tao1986,Tao1997} and its BDCA variant \cite{MR4078808} can be applied for solving problem~\eqref{EQ01}. Although these algorithms only converge to critical points, they can be easily combined as in \cite{Aragon2020} with a basic derivative-free optimization scheme to converge to {d-stationary points}, which satisfy $\partial h(\bar{x})=\{\nabla g(\bar{x})\}$ (or, equivalently, $\varphi'(\bar{x};d)=0$ for all $d\in\mathbb{R}^n$; see~\cite[Proposition~1]{Aragon2020}). In the DC setting, every local minimizer of problem~\eqref{EQ01} is a d-stationary point \cite[Theorem~3]{Toland1979}, a property which is stronger than the notion of stationarity in Definition~\ref{def:stationary}.
\end{remark}\vspace*{-0.1in}

To proceed, recall that a mapping  $f: U\to \mathbb{R}^m$ defined on an open set $U\subseteq \mathbb{R}^n $ is \emph{semismooth} at $\bar{x}$ if it is locally Lipschitzian around $\bar{x}$, directionally differentiable at this point, and the limit
\begin{align*}
\lim\limits_{A \in \tiny{\co} \overline{\nabla} f(\bar{x} + t u'), \atop u' \to u, t\to 0^+} A u'
\end{align*}
exists for all $u \in \mathbb{R}^n$, where $\overline{\nabla} f(x) :=\{ A\;|\; \exists x_k \overset{D}{\to }  x \text{ and } \nabla f(x_k) \to A \}$, and where $D$ is the set on which $f$ is differentiable; see \cite{MR1955649,MR3289054} for more details. We say that a function $g: \mathbb{R}^n \to \Rex$ is {\em semismoothly differentiable} at $\bar{x}$ if $g$ is $\mathcal{C}^{1}$-smooth around $\bar{x}$ and its gradient mapping $\nabla g$ is semismooth at this point.

Recall further that a function $\psi: \mathbb{R}^n \to \Rex$ is \emph{prox-bounded} if there exists $\lambda>0$ such that $\MoreauYosida{\psi}{\lambda}(x)>-\infty$ for some $x\in \mathbb{R}^n$, where
$\MoreauYosida{\psi}{\lambda} :\mathbb{R}^n \to \Rex$  is the \emph{Moreau envelope} of $\psi$ with parameter $\lambda>0$ defined by
\begin{equation}\label{moreau}
\MoreauYosida{\psi}{\lambda} (x):=\inf_{ z \in \mathbb{R}^n }\Big\{ \psi(z) + \frac{1}{2\lambda} \| x-z\|^2\Big\}.
\end{equation}
The number $\lambda_\psi:= \sup \{\lambda >0\;|\;\MoreauYosida{\psi}{\lambda}(x)>-\infty \text{ for some } x\in \mathbb{R}^n \}$  is called the \emph{threshold} of prox-boundedness of $\psi$. The corresponding \emph{proximal mapping} is the multifunction $\Prox{\psi}{\lambda} : \mathbb{R}^n \tto \mathbb{R}^n$ given by
\begin{equation}\label{prox}
\Prox{\psi}{\lambda} (x):= \mathop{\rm argmin}_{z \in \mathbb{R}^n }\Big\{  \psi(z) + \frac{1}{2\lambda} \| x-z\|^2\Big\}.
\end{equation}

Next we observe that that the Moreau envelope can be represented as a {\em DC function}. For any function $\varphi:\mathbb{R}^n\to\Rex$, consider its \emph{Fenchel conjugate} 
\begin{equation*}
\phi^*(x):=\sup_{z\in\mathbb{R}^n}\big\{\langle x,z\rangle-\phi(z)\big\},
\end{equation*}
and for any $\psi\colon\R^n\to\oR$ and $\lm>0$, define the {\em Asplund function}
\begin{equation}\label{asp}
\Asp{\lambda}{\psi}(x):=\sup\limits_{z\in \mathbb{R}^n}\Big\{ \frac{ 1}{\lambda} \langle z,x\rangle - \psi(z)  - \frac{1}{ 2\lambda } \| z\|^2\Big\}=\Big(\psi  + \frac{1}{ 2\lambda } \| \cdot\|^2\Big)^\ast(x),
\end{equation}
which is inspired by Asplund's study of metric projections in \cite{asplund}. The following proposition presents the precise formulation of the aforementioned statement and reveals some remarkable properties of the Asplund function \eqref{asp}.

\begin{proposition}\label{Lemma5.1}
Let $\psi$ be a prox-bounded function with threshold $\lambda_\psi$. Then for every $\lambda \in (0,\lambda_\psi)$, we have the representation
\begin{equation}\label{more-asp}
\MoreauYosida{\psi}{\lambda}(x)=\frac{1}{2\lambda} \| x\|^2 - \Asp{\lambda}{\psi}(x),\quad x\in\R^n,
\end{equation}
where the Asplund function is convex and Lipschitz continuous on $\R^n$.
Furthermore, for any $x\in\R^n$ the following subdifferential evaluations hold:
\begin{align}{ }
\partial(-\Asp{\lambda}{\psi}) (x) &\subseteq-  \frac{1}{\lambda}  \Prox{\psi}{\lambda} (x),\label{eq_sub_eq01}\\
\partial \Asp{\lambda}{\psi}(x) &= \frac{1}{\lambda} \co\left( \Prox{\psi}{\lambda} (x)\right).\label{eq_sub_eq02}
\end{align}
Moreover, if $v\in \Prox{\psi}{\lambda} (x)$ is such that $v\notin \co\left( \Prox{\psi}{\lambda} (x)\backslash \{v\}\right)$, then the vector  $-\frac{1}{\lambda}v$ belongs to $\partial(-\Asp{\lambda}{\psi}) (x)$.
If in addition $f$ is of class $\mathcal{C}^{2,1}$ on $\R^n$, then the function $x\mapsto \Asp{\lambda}{\psi}( x - \lambda \nabla f(x))$ is prox-regular at any point $x\in\R^n$.
\end{proposition}\vspace*{-0.03in}
\begin{proof}
Representation \eqref{more-asp} easily follows from definitions of the Moreau envelope and Asplund function. Due to the second equality in \eqref{asp}, the Asplund function is convex on $\R^n$. It is also Lipschitz continuous due its finite-valuedness on $\R^n$, which is induced by this property of the Moreau envelope. The subdifferential evaluations in \eqref{eq_sub_eq01} and \cite[Example~10.32]{MR1491362} and the subdifferential sum rule in \cite[Proposition~1.30]{MR3823783}) tell us that
$\partial(-\Asp{\lambda}{\psi} )(x) = -\lambda^{-1}x + \partial (	 \MoreauYosida{\psi}{\lambda}) (x)$ and $\partial \Asp{\lambda}{\psi} (x) = \lambda^{-1} x + \partial (-\MoreauYosida{\psi}{\lambda}) (x)$ for any $x\in\R^n$.

Take further $v\in \Prox{\psi}{\lambda} (x) $ with $-\frac{1}{\lambda}v \not\in\partial(-\Asp{\lambda}{\psi}  ) (x)$ and show that $v \in {\rm co}\left( \Prox{\psi}{\lambda} (x)\backslash \{v\}\right)$. Indeed, it follows from \eqref{eq_sub_eq01} that
\begin{align*}
\partial(-\Asp{\lambda}{\psi}  ) (x) &\subseteq-  \frac{1}{\lambda}  \Prox{\psi}{\lambda} (x) \backslash\{v\}.
\end{align*}
The Lipschitz continuity and convexity of $\Asp{\lambda}{\psi}$ implies that 
\begin{align}\label{eqLemma26}
{\rm co}\,\partial(-\Asp{\lambda}{\psi}  ) (x) = -\partial\Asp{\lambda}{\psi} (x)
\end{align}
by \cite[Theorem~3.57]{MR2191744}, which allows us to deduce from \eqref{eq_sub_eq02} and \eqref{eqLemma26} that
\begin{equation*}
{\rm co}\big(\Prox{\psi}{\lambda} (x)\big) = {\rm co}\big( \Prox{\psi}{\lambda} (x) \backslash \{v\}\big).
\end{equation*}
This verifies the inclusion $v \in \co\left( \Prox{\psi}{\lambda} (x)\backslash \{v\}\right)$ as claimed.

Observe finally that the function $x\mapsto \Asp{\lambda}{\psi}_{\lambda }( x - \lambda \nabla f(x))$ is the composition of the convex function $ \Asp{\lambda}{\psi}$ and the $\mathcal{C}^{1,1}$ mapping $x\mapsto  x - \lambda \nabla f(x)$, which ensures by \cite[Proposition~2.3]{MR2069350} its prox-regularity at any  point $x\in \mathbb{R}^n$.
\end{proof}\vspace*{-0.05in}

The following remark discusses a useful representation of the basic subdifferential of the function $-\Asp{\lambda}{\psi}$ and other functions of this type.

\begin{remark}\label{asp-rem}
It is worth mentioning that the subdifferential $\partial (-\Asp{\lambda}{\psi})(x)$ can be expressed via the set $D:=\{x\in\R^n\;|\;\Asp{\lambda}{\psi}\;\mbox{ is differentiable at }\;x\}$ as follows:
\begin{equation}\label{asp-rem1}
\partial (-\Asp{\lambda}{\psi})(x) =\big\{ v\in\R^n\;\big|\;\text{ there exists } x_k \overset{D}{\to} x \text{ and } \nabla \Asp{\lambda}{\psi} (x_k) \to -v\big\}. 
\end{equation}
We refer to \cite[Theorem~10.31]{MR1491362} for more details. Note that we do not need to take the convex hull on the right-hand side of \eqref{asp-rem1} as in the case of the generalized gradient of locally Lipschitzian functions. 
\end{remark}

Finally, recall the definitions of the convergence rates used in the paper.\vspace*{-0.05in}

\begin{definition}\label{def:rates}
Let $\{x_k\}$ be a sequence in $\mathbb{R}^n$ converging to $\bar{x}$ as $k \rightarrow \infty$. The convergence rate is said to be:

{\bf(i)} \emph{R-linear} if there exist $\mu \in(0,1), c>0$, and $k_0 \in \mathbb{N}$ such that
$$
\left\|x_k-\bar{x}\right\| \leq c \mu^k\;\text { for all }\;k \geq k_0.
$$

{\bf(ii)} \emph{Q-linear} if there exists $\mu\in(0,1)$ such that
$$
\limsup_{k\to\infty}\frac{\left\|x_{k+1}-\bar{x}\right\|}{\left\|x_k-\bar{x}\right\|}=\mu.
$$

{\bf(iii)} \emph{Q-superlinear} if it is Q-linear for all $\mu\in (0,1)$, i.e., if
$$
\lim_{k\to\infty}\frac{\left\|x_{k+1}-\bar{x}\right\|}{\left\|x_k-\bar{x}\right\|}=0.
$$

{\bf(iv)} \emph{Q-quadratic} if we have
$$
\limsup_{k\to\infty}\frac{\left\|x_{k+1}-\bar{x}\right\|}{\left\|x_k-\bar{x}\right\|^2}<\infty.
$$
\end{definition}\vspace*{-0.2in}

\section{Regularized Coderivative-Based Damped Semi-Newton  Method in Nonsmooth Difference Programming}\label{sec:3}\vspace*{-0.05in}

The goal of this section is to justify the well-posedness and good performance of the novel algorithm RCSN under appropriate and fairly general assumptions. In the following remark, we discuss the difference between the choice of subgradients and hence of directions in RCSN and DC algorithms.\vspace*{0.05in}

Our main RCSN algorithm to find stationary points of nonsmooth problems \eqref{EQ01} of difference programming is labeled below as Algorithm~\ref{alg:1}.

\begin{algorithm}[h!]
\begin{algorithmic}[1]
\Require{$x_0 \in \R^n$, $\beta \in (0,1)$, $\zeta>0$, $t_{\min}>0 $, $\rho_{\max}>0$ and $\sigma\in(0,1)$.}
\For{$k=0,1,\ldots$}
\State Take $w_k\in  \partial \varphi (x_k)$. If $  w_k=0$, STOP and return $x_k$. 
\State Choose  $\rho_k\in[0,\rho_{\max}]$ and $d_k \in \mathbb{R}^n\backslash \{ 0\}$ such that
\begin{align}
-w_k\in \partial^2 g(x_k)(d_k)+\rho_kd_k\quad\text{and}\quad \langle w_k,d_k\rangle\leq -\zeta\|d_k\|^2. \label{EQALG01}
\end{align}
\State Choose any $\overline{\tau}_k\geq t_{\min}$. Set $\overline{\tau}_k:=\tau_k$.
\While{$\varphi(x_k + \tau_k d_k) > \varphi(x_k) +\sigma \tau_k \langle w_k ,   d_k\rangle $}
\State $\tau_k = \beta \tau_k$.
\EndWhile
\State Set $x_{k+1}:=x_k + \tau_kd_k$. \label{step5}
\EndFor
\end{algorithmic}
\caption{Regularized coderivative-based damped semi-Newton  algorithm for  nonsmooth difference programming}\label{alg:1}
\end{algorithm}\vspace*{-0.05in}

\begin{remark}\label{rem:subgr}
Observe that Step~2 of Algorithm~\ref{alg:1} selects $w_k\in\partial\varphi(x_k)=\nabla g(x_k)+\partial(-h)(x_k)$, which is equivalent to choosing $v_k:=w_k-\nabla g(x_k)$ in the basic subdifferential of $-h$ at $x_k$. Under our assumptions, the set $\partial(-h)(x_k)$ can be {\em considerably smaller} than $\partial h(x_k)$; see the proof of Proposition~\ref{Lemma:Dire01} and also Remark~\ref{asp-rem} above. Therefore, Step~2 differs from those in DC algorithms, which choose subgradients in $\partial h(x_k)$. The purpose of our development is to find a {\em stationary point} instead of a (classical) critical point for problem~\eqref{EQ01}.
In some applications, Algorithm~\ref{alg:1} would not be implementable if the user only has access to subgradients contained in $\partial h(x_k)$ instead of $\partial(-h)(x_k)$. In such cases, a natural alternative to Algorithm~\ref{alg:1} would be a scheme replacing $w_k\in\partial\varphi(x_k)$ in Step~2 by $w_k:=\nabla g(x_k)+v_k$ with $v_k\in\partial h(x_k)$. Under the setting of our convergence results, the modified algorithm would find a critical point for problem~\eqref{EQ01}, which is not guaranteed to be stationary.
\end{remark}\vspace*{-0.1in}

The above discussions are illustrated by the following example.\vspace*{-0.1in}

\begin{example}
Consider problem~\eqref{EQ01} with $g(x):=\frac{1}{2}x^2$ and $h(x):= |x|$. If an algorithm similar to Algorithm~\ref{alg:1} was run by using $x_0 =0$ as the initial point but choosing $w_0=\nabla g(x_0)+v_0$ with $v_0= 0 \in\partial h(0)$ (instead of $w_0\in\partial\varphi(x_0)$), it would stop at the first iteration and return $x=0$, which is a critical point, but not a stationary one. On the other hand, for any $w_0\in\partial \varphi(0)=\{-1,1\}$ we get $w_0\neq 0$, and so Algorithm~\ref{alg:1} will continue iterating until it converges to one of the two stationary points $-1/2$ and $1/2$, which is guaranteed by our main convergence result;  see Theorem~\ref{The01} below.
\end{example}\vspace*{-0.1in}
  	
The next lemma shows that Algorithm~\ref{alg:1} is well-defined by proving 
the existence of a direction $d_k$ satisfying \eqref{EQALG01} in Step~3 for sufficiently large regularization parameters $\rho_k$.\vspace*{-0.1in}

\begin{lemma}\label{lemma1}
Let $\varphi: \mathbb{R}^n \to \R$ be the objective function in problem~\eqref{EQ01} with $g\in\mathcal{C}^{1,1}$ and $h$ being locally Lipschitz around $\ox$ and prox-regular at this point. Further, assume that $\partial^2 g(\bar{x})$ is $\xi$-lower-definite for some $\xi\in\mathbb{R}$ and consider a nonzero subgradient $w\in \partial \varphi(\bar{x})$.
Then for any $\zeta>0$ and any $\rho\geq\zeta-\xi$, there exists a nonzero direction $d\in \mathbb{R}^n$ satisfying the inclusion
\begin{align}\label{Eq002}
-w \in\partial^2 g(\bar{x})(d)+\rho d.
\end{align}\vspace*{-0.05in}
Moreover, any nonzero direction from \eqref{Eq002} obeys the conditions:\\[1ex]
{\bf(i)}\label{lemma1b} $\varphi'(\bar{x}; d) \leq \langle w, d\rangle  \leq -\zeta\| d\|^2$.\\[1ex]
{\bf(ii)} \label{lemma1c} Whenever $\sigma \in (0,1)$, there exists $\eta >0$ such that
\begin{align*}
\varphi(\bar{x} + \tau d ) < \varphi(\bar{x})  + \sigma \tau \langle w, d\rangle \leq  \varphi(\bar{x})- \sigma \zeta  \tau \|d\|^2\;\mbox{ when }\;\tau \in (0,\eta).
\end{align*}
\end{lemma}
\begin{proof}
Consider the function $\psi(x):=g(x)+\langle w-\nabla g(\bar{x}),x\rangle+\frac{\rho}{2}\|x\|^2$ for which  we clearly have that $\partial^2 \psi(\bar{x}) =\partial^2 g(\bar{x})+\rho I$, where $I$ denotes the identity mapping. This shows by Remark~\ref{rem:definite} that $\partial^2 \psi(\bar{x})$ is $(\xi+\rho)$-lower-definite, and thus it is $\zeta$-lower-definite as well. Since $\nabla \psi(\bar{x})=w\neq 0$ and $\zeta>0$, it follows from \cite[Proposition~3.1]{2021arXiv210902093D} (which requires $\psi$ to be $\mathcal{C}^{1,1}$ on $\R^n$, but actually only $\mathcal{C}^{1,1}$ around $\bar{x}$ is needed) that there exists a nonzero direction $d$ such that  	 $-\nabla \psi(\bar{x})\in \partial^2 \psi (\bar{x})(d)$. This readily verifies \eqref{Eq002}, which yields in turn the second inequality in (i) due to Definition~\ref{def:lower-def}. On the other hand, we have by Proposition~\ref{Lemma:Dire01} the following:
 \begin{align}\label{lemma1:INQ001}
 \begin{aligned}
 \varphi'(\bar{x}; d)=\lim\limits_{t \to 0^+}\frac{ g(\bar{x} +td) -g(\bar{x}) }{ t  }+\limsup\limits_{t \to 0^+}\frac{ -h(\bar{x} +td) +h(\bar{x}) }{t } \\
 = \langle \nabla g(\bar{x}), d \rangle  +  \inf\big\{\langle w, d\rangle\;\big|\;w\in -\partial h(\bar x)\big\}\leq \langle \nabla g(\bar{x}) + v, d\rangle \leq- \zeta \|d\|^2,
 \end{aligned}
 \end{align}
where in the last estimate is a consequence of the second inequality in (i).

Finally, assertion (ii) follows directly from \eqref{lemma1:INQ001} and the definition of directional derivatives \eqref{UDD}.
\end{proof}\vspace*{-0.2in}

\begin{remark} Under the $\xi$-lower-definiteness of $\partial^2 g(x_k)$, Lemma~\ref{lemma1} guarantees the existence of a direction $d_k$ satisfying both conditions in~\eqref{EQALG01} for all $\rho_k\geq \zeta-\xi$. {\em When $\xi$ is unknown}, it is still possible to implement Step~3 of the algorithm as follows. Choose first any initial value of $\rho\geq 0$, then compute a direction satisfying the inclusion in \eqref{EQALG01} and continue with Step~4 if the descent condition in \eqref{EQALG01} holds. Otherwise, increase the value of $\rho$ and repeat the process until the descent condition is satisfied.
\end{remark}\vspace*{-0.05in}

The next example demonstrates that the {\em prox-regularity} of $h$ is {\em not a superfluous assumption} in Lemma~\ref{lemma1}. Namely, without it the direction $d$ used in Step~3 of Algorithm~\ref{alg:1} can even be an {\em ascent direction}.\vspace*{-0.03in}

\begin{example}\label{ex:failure}
Consider the {\em least squares problem} given by
$$
\min_{x\in\mathbb{R}^2} \frac{1}{2}(Ax-b)^2+\| x\|_1 -  \|x\|_2,\quad x\in\R^2,
$$
with $A:=[1,0]$ and $b:=1$. Denote $g(x):=\frac{1}{2}\|Ax-b\|^2$ and $h(x):=\| x\|_2 -  \|x\|_1$. If we pick $\bar{x}:=(1,0)^\tr$, the function $h$ is not prox-regular at $\bar{x}$ because it is not lower regular at $\bar{x}$; see Proposition~\ref{Lemma:Dire01}. Indeed, $\Hat\partial h(\bar{x})=\emptyset$, while
$$
\partial h(\bar{x})=\frac{\bar{x}}{\|\bar{x}\|}+\partial(-\|\cdot\|_1)(\bar{x})=\left\{\begin{pmatrix} 0\\ -1  \end{pmatrix},\begin{pmatrix} 0\\ 1  \end{pmatrix}\right\}.
$$
Therefore, although $\nabla^2 g(\bar{x})=A^TA$ is $\lambda_{\min}(A^TA)$-lower-definite, the assumptions of Lemma~\ref{lemma1} are not satisfied. Due to the representation
$$
\partial (-h)(\bar{x})=-\frac{\bar{x}}{\|\bar{x}\|}+\partial\|\cdot\|_1(\bar{x})=\left\{\begin{pmatrix} 0\\ v  \end{pmatrix}\;\Bigg|\;v\in [-1,1]\right\},
$$
the choice of $v:=(0, 1)^\tr\in\partial (-h)(\bar{x})$ yields $w:=\nabla g(\bar{x})+v = (0, 1)^\tr\in\partial\varphi(\bar{x})$. For any $\rho>0$, inclusion \eqref{Eq002} gives us $d = (0,- 1/\rho)^\tr$. This is an ascent direction for the objective function $\varphi(x)=g(x)-h(x)$ at $\bar{x}$ due to
$$
\varphi(\bar{x}+\tau d)=1+\frac{\tau}{\rho}-\sqrt{1+(\tau/\rho)^2}>\varphi(\bar{x})=0\;\mbox{ for all }\;\tau>0,
$$
which illustrates that the prox-regularity is an essential assumption in Lemma~\ref{lemma1}.
\end{example}\vspace*{-0.07in}

Algorithm~\ref{alg:1} either stops at a stationary point, or produces an infinite sequence of iterates. The convergence properties of the iterative sequence of our algorithm are obtained below in the main theorem of this section. Prior to the theorem, we derive yet another lemma, which establishes the following {\em descent property} for the difference of a $\mathcal{C}^{1,1}$  function and a prox-regular one.\vspace*{-0.05in}

 \begin{lemma}\label{Lemma:01}
 Let $\varphi(x) = g(x) -h(x)$, where $g$ is of class $\mathcal{C}^{1,1}$ around $\ox$, and where $h$ is continuous around $\ox$ and prox-regular at this point. Then for every $\bar{v} \in \partial h(\bar{x})$, there exist positive numbers $\epsilon$ and $r$ such that
\begin{align*}
 \varphi(y) \leq \varphi(x) + \langle \nabla g(x) - v , y-x \rangle +r \| y-x\|^2
\end{align*}
whenever $x, y \in \mathbb{B}_{\epsilon}(\bar{x})$ and $v\in \partial h(x) \cap \mathbb{B}_{\epsilon}(\bar{v}) $. 	
\end{lemma}\vspace*{-0.05in}
\begin{proof}
Pick any $\bar{v} \in \partial h(\bar{x})$ and deduce from the imposed prox-regularity and continuity of $h$ that there exist $\epsilon_1>0$ and $r_1>0$ such that
\begin{align}\label{Lemma:01:Eq01}
-h(y) \leq -h(x) + \langle - v , y - x \rangle + r_1 \| y-x\|^2\;\mbox{ for all }\;x,y \in \mathbb{B}_{\epsilon_1}(\bar{x})
 \end{align}
and all $v\in \partial h(x)\cap \mathbb{B}_{\epsilon_1}(\bar{v}) $. It follows from the $\mathcal{C}^{1,1}$ property of $g$ by \cite[Lemma~A.11]{MR3289054}  that there exist positive numbers $r_2$ and $\epsilon_2$ such that
\begin{align}\label{Lemma:01:Eq02}
 g(y) \leq g(x) + \langle \nabla g(x) , y-x \rangle + r_2 \| y- x\|^2\;\mbox{ for all }\;\mathbb{B}_{\ve_2}.
 \end{align}
 Summing up the inequalities in \eqref{Lemma:01:Eq01} and \eqref{Lemma:01:Eq02} and defining $r:= r_1 +r_2$ and $\epsilon := \min\{ \epsilon_1,\epsilon_2\}$, we get that
 \begin{align*}
 g(y)- h(y) \leq g(x) - h(x) + \langle \nabla g(x)-v , y-x \rangle + r\| y-x\|^2
 \end{align*}
for all $x,y \in \mathbb{B}_\epsilon(\bar{x})$ and all $v\in \partial h(x) \cap \mathbb{B}_\epsilon(\bar{v})$. This completes the proof.
 \end{proof}\vspace*{-0.1in}

Now we are ready to establish the aforementioned theorem about the performance of Algorithm~\ref{alg:1}.\vspace*{-0.05in}
 
\begin{theorem}\label{The01}
Let $\varphi: \mathbb{R}^n \to \Rex$ be the objective function of   problem~\eqref{EQ01} given by $\varphi =  g-h$ with $\inf \varphi >-\infty$. Pick an initial point $x_0 \in \mathbb{R}^n$  and  suppose that the sublevel set $\Omega:=\{x \in\mathbb{R}^n\;|\;\varphi(x) \leq \varphi(x_0)\}$ is closed. Assume also that:\\[0.5ex]
{\bf(a)} \label{Theo01ass:a} The function $g$ is $\mathcal{C}^{1,1}$ around every $x\in\Omega$ and the second-order subdifferential $\partial^2 g(x)$ is $\xi$-lower-definite for all $x\in \Omega$ with some $\xi\in\mathbb{R}$.\\[1ex]
{\bf(b)} The function $h$ is locally Lipschitzian and prox-regular on $\Omega$.\\[0.5ex]
Then Algorithm~{\rm\ref{alg:1}} either stops at a stationary point, or produces sequences $\{x_k\} \subseteq \Omega$, $\{\varphi(x_k)\}$, $\{w_k\}$, $\{d_k\}$, and $\{\tau_k\}$ such that:\\[0.5ex]
{\bf(i)} \label{The01a} The sequence $\{\varphi(x_k)\}$ monotonically decreases and converges.\\[0.5ex]
{\bf(ii)}\label{The01b} If $\{x_{k_j}\}$ as $j\in \N$ is any bounded subsequence of $\{x_k\}$, then $\disp\inf_{j\in \N}\tau_{k_j}>0$,
\begin{align*}
\sum\limits_{j \in \N}\| d_{k_j}\|^2 < \infty,\; \sum\limits_{j\in \N } \| x_{k_j +1} - x_{k_j}\|^2< \infty,\;\text{ and }\;\sum\limits_{j\in \N } \|w_{k_j} \|^2<\infty.
\end{align*}
In particular, the boundedness of the entire sequence $\{x_k\}$ ensures that the set of accumulation points of $\{x_k\}$ is a nonempty, closed, and connected.\\[0.5ex]
{\bf(iii)}\label{The01c} If $x_{k_j} \to \bar{x}$ as $j\to\infty$, then $\bar{x}$ is a stationary  point of problem \eqref{EQ01} with the property $\varphi(\bar{x}) =\disp\inf_{k\in \N} \varphi (x_k)$.\\[0.5ex]
{\bf(iv)}\label{The01d} If $\{x_k\}$ has an isolated accumulation point $\bar{x}$, then the entire sequence $\{x_k\}$ converges to $\bar{x}$ as $k\to\infty$, where $\bar x$ is a stationary point of \eqref{EQ01}.	
\end{theorem}\vspace*{-0.05in}
\begin{proof}
If Algorithm~\ref{alg:1} stops after a finite number of iterations, then it clearly returns a stationary point. Otherwise, it produces an infinite sequence $\{x_k\}$. By Step~5 of Algorithm~\ref{alg:1} and Lemma~\ref{lemma1}, we have that $\inf \varphi\le\varphi(x_{k+1}) <  \varphi(x_k) $ for all $k \in \N$, which proves assertion (i) and also shows that $\{x_k\} \subseteq \Omega$.

To proceed, suppose that $\{x_k\}$ has a bounded subsequence 
$\{x_{k_j}\}$ (otherwise there is nothing to prove) and split the rest of the proof into the {\em five claims}.\vspace*{0.03in}
	
\noindent\textbf{Claim~1:} \emph{The sequence $\{\tau_{k_j}\}$, associated with $\{x_{k_j}\}$ as $j\in\N$ and produced by Algorithm~{\rm\ref{alg:1}}, is bounded from below.}\\
Indeed, otherwise consider  a subsequence $\{\tau_{\nu_i}\}$ of $\{\tau_{k_j}\}$ such that $\tau_{ \nu_i} \to 0^+$ as $i\to \infty$. Since $\{x_{k_j}\}$ is bounded, we can assume that $\{x_{\nu_i}\} $ converges to some point $\bar{x}$. By Lemma~\ref{lemma1}, we have that
\begin{align}\label{EQ002}
-\langle w_{\nu_i}, d_{\nu_i} \rangle \geq \zeta \|d_{\nu_i}\|^2\;\mbox{ for all }\;i\in\N,
\end{align}
which yields by the Cauchy--Schwarz inequality the estimate
\begin{align}\label{EQ002bis}
\|w_{\nu_i}\|\geq \zeta \| d_{\nu_i}\|,\quad i\in\N.
\end{align}
Since $\varphi$ is locally Lipschitzian and $w_{\nu_i} \in \partial\varphi (x_{{\nu_i}})$, we suppose without loss of generality that $w_{\nu_i}$ converges to some $\bar{w} \in \partial\varphi(\bar{x}) \subseteq  \nabla g(\bar{x})- \partial h(\bar{x})$ as $i\to\infty$. It follows from \eqref{EQ002bis} that $\{d_{\nu_i}\}$ is bounded, and therefore $d_{\nu_i} \to\bar{d}$ along a subsequence. Since $\tau_{ \nu_i } \to 0^+$, we can assume that $\tau_{ \nu_i }<t_{\min}$ for all $i\in\N$, and hence Step~5 of Algorithm~\ref{alg:1} ensures the inequality
\begin{align}\label{EQ003}
\varphi(x_{\nu_i} + \beta^{-1}\tau_{\nu_i} d_{\nu_i}) > \varphi(x_{\nu_i}) +\sigma \beta^{-1}\tau_{\nu_i} \langle w_{\nu_i},d_{\nu_i}\rangle,\quad i\in\N.
\end{align}
Lemma~\ref{Lemma:01}  gives us a constant $r>0$ such that 
\begin{align}\label{EQ004}
\varphi(x_{\nu_i} + \beta^{-1}\tau_{\nu_i} d_{\nu_i})\le\varphi(x_{\nu_i}) +\beta^{-1}\tau_{\nu_i} \langle w_{\nu_i} ,
d_{\nu_i}\rangle+r\beta^{-2}\tau_{\nu_i}^2\|d_{\nu_i}\|^2
\end{align}
for all $i$ sufficiently large. Combining \eqref{EQ003},~\eqref{EQ004}, and \eqref{EQ002} tells us that
\begin{equation*}
\begin{array}{ll}
\sigma \beta^{-1}\tau_{\nu_i} \langle w_{\nu_i} ,  d_{\nu_i}\rangle<	\varphi(x_{\nu_i} + \beta^{-1}\tau_{\nu_i} d_{\nu_i}) - \varphi(x_{\nu_i})\\
\leq \beta^{-1}\tau_{\nu_i} \langle w_{\nu_i} ,   d_{\nu_i}\rangle+r\beta^{-2}\tau_{\nu_i}^2\|d_{\nu_i}\|^2\leq \beta^{-1}\tau_{\nu_i}\left(1-\disp\frac{r}{\zeta\beta}\tau_{\nu_i}\right)\langle w_{\nu_i}, d_{\nu_i}\rangle
\end{array}
\end{equation*}
for large $i$. Since $\langle w_{\nu_i} ,d_{\nu_i}\rangle<0$ by \eqref{EQ002},  we get that $\sigma >1 - \frac{r}{ \zeta\beta  } \tau_{ \nu_i }$
for such $i$, which contradicts the choice of $\sigma \in (0,1)$ and thus verifies this claim.\vspace*{0.03in}

\noindent\textbf{Claim~2:}  \emph{We have the series convergence $\sum_{j \in \N} \| d_{k_j}\|^2 < \infty $, $\sum_{j\in \N } \| x_{k_j +1} - x_{k_j}\|^2 < \infty$,  and $\sum_{j\in \N } \| w_{k_j}\|^2 < \infty$.}\\
To justify this, deduce from Step~5 of Algorithm~\ref{alg:1} and Lemma~\ref{lemma1} that
\begin{align*}
\sum\limits_{k\in \N}  \zeta\tau_k \| d_k\|^2 \leq \frac{1}{\sigma   }  \Big( \varphi(x_0) - \inf_{k\in \N} \varphi(x_k)\Big).
\end{align*}
It follows from Claim~1 that $\zeta\tau_{k_j} >\gamma >0$ for all $j\in \N$, which yields $\sum_{j\in \N}\| d_{k_j}\|^2  <\infty$. On the other hand, we have that $\| x_{k_j +1} - x_{k_j}\|= \tau_{k_j}\| d_{k_j}\|$, and again Claim~1 ensures that $\sum_{j\in \N } \| x_{k_j +1} - x_{k_j}\|^2 < \infty$. To proceed further, let $l_2 := \sup\{ \| d_{k_j}\|\;|\;\in \N \}$ and use the Lipschitz continuity of $\nabla g$ on the compact set ${\rm cl}\{x_{k_j}\;|\;{j\in \N}\} \subseteq \Omega$. Employing the subdifferential condition from \cite[Theorem~4.15]{MR3823783} together with the coderivative scalarization in \eqref{coder:sub}, we get by the standard compactness argument the existence of $l_3>0$ such that
\begin{align*}
w \in \partial \langle d, \nabla g\rangle (x_{k_j})= \partial^2 g(x_{k_j})(d)  \Longrightarrow \| w\| \leq l_3
\end{align*}
for all $j\in\N$ and all $d\in \mathbb{B}_{l_2}(0)$. Therefore, it follows from the inclusion $-w_{k_j} \in \partial^2 g(x_{k_j})(d_{k_j}) +\rho_{k_j}d_{k_j}$  that we have
\begin{align}\label{LipGrad}
\|w_{k_j} +\rho_{k_j}d_{k_j}\| \leq l_3 \| d_{k_j}\|\;\text{ for all large }\;j \in \N.
\end{align}
 Using finally the triangle inequality and the estimate $\rho_k\leq \rho_{\max}$ leads us to the series convergence $\sum_{j\in \N } \| w_{k_j} \|^2 < \infty$ as stated in Claim~2.\vspace*{0.03in}
	
 \noindent\textbf{Claim~3:} \emph{If the sequence $\{x_k\}$ is bounded, then the set of its accumulation points is nonempty, closed and connected.}\\
Applying Claim~2 to the sequence $\{x_k\}$, we have the \emph{Ostrowski condition} $\lim_{k \to \infty }\| x_{k +1} - x_{k}\| = 0$. Then, the conclusion follows from \cite[Theorem~28.1]{Ostrowski1966}.

\noindent\textbf{Claim~4:} \emph{If $x_{k_j} \to \bar{x}$ as $j\to\infty$, then $\bar{x}$ is a stationary point of \eqref{EQ01} being such that $\varphi(\bar{x}) = \inf_{k\in \N} \varphi (x_k)$.} \\
By Claim 2, we have that the sequence $w_{k_j} \in \partial \varphi(x_{k_j})$ with $w_{k_j} \to 0$ as $j\to\infty$. The closedness of the basic subgradient set ensures that $0 \in \partial \varphi (\bar{x})$. The second assertion of the claim follows from the continuity of $\varphi$ at $\bar{x} \in \Omega$.

\noindent\textbf{Claim~5:} \emph{If $\{x_k\}$ has an isolated accumulation point~$\bar{x}$, then the entire sequence of $x_k$ converges to $\bar{x}$ as $k\to\infty$, and $\ox$ is a stationary point of \eqref{EQ01}.}
Indeed, consider any subsequence $x_{k_j} \to \bar{x}$. By Claim~4, $\bar{x}$ is a stationary point of \eqref{EQ01}, and it follows from Claim~2 that
$\lim_{j\to \infty}  \| x_{k_j +1}-x_{k_j}\|=0$. Then we deduce from by \cite[Proposition~8.3.10]{MR1955649} that $x_k \to \bar{x}$ as $k\to\infty$, which completes the proof of theorem.
\end{proof}\vspace*{-0.2in}

\begin{remark}\label{rem:theorem}
Regarding Theorem~\ref{The01}, observe the following:

(i) If $h=0$, $g$ is of class $\mathcal{C}^{1,1}$, and $\xi>0$, then the results of Theorem~\ref{The01} can be found in \cite{2021arXiv210902093D}.

(ii) If $\xi\geq 0$, we can choose the regularization parameter $\rho_k:=c\|w_k\|$ and (a varying) $\zeta:=c\|w_k\|$ in~\eqref{EQALG01} for some $c>0$ to verify that assertions (i) and (iii) of Theorem~\ref{The01} still hold. Indeed, if $\{x_{k_j}\}$ converges to some $\bar{x}$, then  $\{w_{k_j}\}$ is bounded by the Lipschitz continuity of $\ph$. Hence the sequence $\{w_{k_j}\}$ converges to $0$. Otherwise, there exists $M>0$  and a subsequence of $\{w_{k_j}\}$ whose norms are bounded from below by $M$. Using the same argumentation as in the proof of Theorem~\ref{The01} with $\zeta=c M$, we arrive at the contradiction with $0$ being an accumulation point of of $\{w_{k_j}\}$.
\end{remark}\vspace*{-0.05in}

When the objective function $\varphi$ is coercive and its stationary points are  isolated, Algorithm~\ref{alg:1} converges to a stationary point because Theorem~\ref{The01}(iii) ensures that the set of accumulation points is connected. This property enables us to prove the convergence in some settings when even there exist nonisolated accumulation points; see the two examples below.\vspace*{-0.05in}

\begin{example}
Consider the function $\varphi: \mathbb{R} \to \R$ given by
\begin{align*}
\varphi(x) & :=\int_0^x t^4 \sin\left(\frac{\pi}{t}\right) dt.
\end{align*} 
This function is clearly $\mathcal{C}^2$-smooth and coercive. For any starting point $x_0$, the level set $\Omega =\{ x\;|\;\varphi(x) \leq \varphi(x_0) \}$ is bounded, and hence there exists a number $\xi\in \R$ such that the functions $g(x) :=\varphi(x)$  and $h(x):=0$ satisfy the assumptions of Theorem~\ref{The01}. Observe furthermore that $\varphi$ is a DC function because it is $\mathcal{C}^2$-smooth; see, e.g., \cite{Oliveira_2020,hiriart}. However, it is not possible to write its DC decomposition with $g(x) = \varphi(x) + ax^2$ and $h(x)=ax^2$ for $a>0$, since there exists no  scalar $a>0$ such that the function $g(x) = \varphi(x) + ax^2$ is convex on the entire real line.

It is easy to see that the stationary points of $\varphi$ are described by $S:=\left\{ \frac{1}{n}\;\big|\;n \in\mathbb{Z}\backslash\{ 0\} \right\}\cup\{ 0\}$. Moreover, if Algorithm~\ref{alg:1} generates an iterative sequence $\{x_k\}$ starting from $x_0$, then the accumulation points form by Theorem~\ref{The01}(ii) a nonempty, closed, and connected set $A \subseteq S$ .
If $A=\{ 0\}$, the sequence $\{x_k\}$ converges to $\bar{x}=0$. If $A$ contains any point of the form $\bar{x}=\frac{1}{n}$, then it is an isolated point, and Theorem~\ref{The01}(iv) tells us that the entire sequence $\{x_k\}$ converges to that point, and consequently we have $A=\{\bar{x}\}$.
\end{example}

\begin{example}\label{example3.10}
Consider the function $\varphi: \mathbb{R}^n \to \R$ given by
\begin{align*}
\varphi(x):=\sum_{i=1}^n \varphi_i(x_i),\; \text{ where }\;\varphi_i(x_i):= g_i(x_i) - h_i(x_i) \\
\text{ with }\;g_i(x_i):= \frac{1}{2}x_i^2\; \text{ and }\;h_i(x_i):= |x_i| +\big| 1-|x_i|\,\big|.
\end{align*}
We can easily check that the function $\varphi$ is coercive and satisfies the  assumptions of Theorem~\ref{The01} with $g(x):=\sum_{i=1}^n g_i(x_i)$, $h(x):= \sum_{i=1}^n h_i(x_i)$, and $\xi=1$. For this function, the points in the set $\{-2, -1,0,1,2\}^n$ are critical but not stationary. Moreover, the points in the set $\{-2,0,2 \}^n$ give the global minima to the objective function $\ph$. Therefore, Algorithm~\ref{alg:1} leads us to global minimizers of $\ph$ starting from any initial point.

 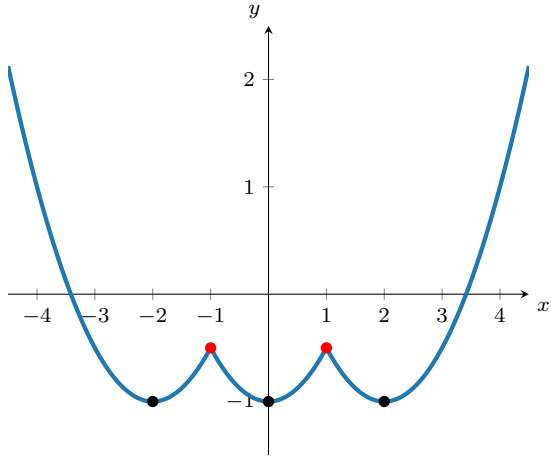
\begin{figure}[h!!]
\centering\begin{tikzpicture}
\begin{axis}[
axis x line=center,
axis y line=center,
xtick={-5,-4,...,5},
ytick={-5,-4,...,4},
xlabel={$x$},
ylabel={$y$},
xlabel style={below right},
ylabel style={above left},
xmin = -4.5, xmax = 4.5,
ymin = -1.5, ymax = 2.5]
\addplot[
domain = -4.5:4.5,
samples = 200,
smooth,
ultra thick,
color=C0,
] {0.5*x*x  - abs(x)- abs(  1-abs(x))  )};
\node at (0,-1)[circle,fill,inner sep=1.5pt]{};
\node at (2,-1)[circle,fill,inner sep=1.5pt]{};
\node at (-2,-1)[circle,fill,inner sep=1.5pt]{};
\node at (-1,-0.5)[circle,fill=red,inner sep=1.5pt]{};
\node at (1,-0.5)[circle,fill=red,inner sep=1.5pt]{};
\end{axis}
\end{tikzpicture}
\caption{ Plot of the function  $\varphi_i $ in Example \ref{example3.10}}
\label{fig:screenshot001}
\end{figure}
\end{example}\vspace*{-0.1in}

The following theorem establishes convergence rates of the iterative sequences in Algorithm~\ref{alg:1} under some additional assumptions.\vspace*{-0.05in}

\begin{theorem}\label{corSMR}
Suppose in addition to the assumptions of Theorem~{\rm\ref{The01}}, that $\{x_k\}$ has an accumulation point $\bar{x}$ such that the subgradient mapping $\partial \varphi$ is strongly metrically subregular at $(\bar{x},0)$. Then the entire sequence $\{x_k\}$ converges to $\bar{x}$ with the Q-linear convergence rate for $\{\varphi(x_k)\}$ and the R-linear convergence rate for $\{x_k\}$ and $\{w_k\}$. If furthermore, $\xi>0$, $0<\zeta\leq\xi$, $\rho_k\to 0$, $\sigma \in(0,\frac{1}{2})$, $t_{\min}=1$, $g$ is semismoothly differentiable at $\bar{x}$, $h$ is  of class $\mathcal{C}^{1,1}$ around $\bar{x}$, and $\clm \nabla h(\bar{x})=0$, then the rate of convergence of all the sequences above is at least Q-superlinear. 
\end{theorem}\vspace*{-0.05in}
\begin{proof} We split the proof of the theorem into the following two claims.\vspace*{0.03in}

\noindent\textbf{Claim~1:} \emph{The rate of convergence of
$\{\varphi(x_k)\}$ is at least Q-linear, while both sequences $\{x_k\}$ and $\{w_k\}$ converge at least R-linearly.}\\
Observe first that it follows from the imposed strong metric subregularity of $\partial\ph$ that $\bar{x}$ is an isolated accumulation point, and so $x_k\to\bar{x}$ as $k\to\infty$ by Theorem~\ref{The01}(iii). Further, we get from \eqref{def:stron_subreg} that there exists $\kappa>0$ such that
\begin{align}\label{eq:1}
\| x_k - \bar{x}\| \leq \kappa \| w_k\|\;\text{ for large }\;k \in \N,
\end{align}
since $w_k\to 0$ as $k\to\infty$ by Theorem \ref{The01}(ii). Using \eqref{LipGrad} and the triangle inequality gives us $\ell > 0$ such that $ \| w_k\| \leq \ell \| d_k\| $ for sufficiently large $k\in\N$. Lemma~\ref{Lemma:01} yields then the cost function increment estimate
\begin{align}\label{eq:2}
\varphi(x_k)-\varphi(\bar{x}) \le r\| x_k -\bar x\|^2\;\text{ for all large }\;k \in\N.
\end{align}
By Step~5 of Algorithm~\ref{alg:1} and Lemma~\ref{lemma1}, we get that  $\varphi(x_k)-\varphi(x_{k+1}) \geq  \sigma \zeta\tau_k \| d_k \|^2 $ for large $k \in \N$. Remembering that $\inf_{k\in \N} \tau_k >0$, we deduce from Theorem \ref{The01}(ii) the existence of $\eta >0$ such that
\begin{align}\label{eq:3}
\varphi(x_k)-\varphi(\bar{x}) - (\varphi(x_{k+1})-	\varphi(\bar{x}) ) \geq \eta  \| w_k\|^2
\end{align}
whenever $k$ large enough. Therefore, applying \cite[Lemma~7.2]{2021arXiv210902093D} to the sequences $\alpha_k := \varphi(x_k) - \varphi(\bar{x})$, $\beta_k:= \|  w_k\|$, and $\gamma_k := \| x_k - \bar{x}\|$ with the positive constants $c_1:= \eta$, $c_2:= \kappa^{-1}$, and $c_3:= r$, we verify the claimed result.\vspace*{0.03in}

\noindent\textbf{Claim ~2: } \emph{Assuming that $\sigma \in (0,\frac{1}{2})$, $t_{\min}=1$, $g$ is semismoothly differentiable at $\bar{x}$, $h$ is of class $\mathcal{C}^{1,1}$ around $\bar{x}$, and $\clm \nabla h(\bar{x})=0$, we have that  the rate of convergence for all the above sequences is at least Q-superlinear.}\\
Suppose without loss of  generality that $h$ is differentiable at any $x_k\to\ox$. It follows from the coderivative scalarization \eqref{coder:sub} and the basic subdifferential sum rule in \cite[Theorem~2.19]{MR3823783} valid under the imposed assumptions that
\begin{align}\label{sublineal}
\partial^2 g(x_k)(d_k) \subseteq \partial^2 g(x_k) (x_k + d_k - \bar{x})  + \partial^2 g(x_k) (-x_k +  \bar{x}).
\end{align}
This yields the existence of $z_k \in \partial^2 g(x_k) (-x_k +  \bar{x})+\rho_k(-x_k+\bar{x})$ such that 
\begin{align}\label{inclusion01}
-\nabla g(x_k) +\nabla h(x_k) - z_k \in  \partial^2 g(x_k) (x_k + d_k - \bar{x})+\rho_k(x_k + d_k - \bar{x}).
\end{align}
Moreover, the  $(\xi+\rho_k)$-lower-definiteness of $\partial^2 g(x_k )+\rho_kI$ and the Cauchy--Schwarz inequality imply that
 \begin{align*}
\|  x_k + d_k -\bar{x}\|\leq \frac{1}{\xi+\rho_k} \|  \nabla g(x_k) - \nabla h(x_k)+ z_k\|.
\end{align*}
Combining now the semismoothness of $\nabla g$ at $\bar{x}$ with the conditions
$\nabla g(\bar{x})=\nabla h(\bar{x})$ and $\clm \nabla h(\bar{x})=0$ brings us to the estimates
\begin{align*}
\begin{array}{ll}
\| \nabla g(x_k) - \nabla h(x_k)+ z_k \|\leq  \| \nabla g(x_k) -\nabla g(\bar{x})+ z_k+\rho_k(x_k-\bar{x})\|\\
+\rho_k\|x_k-\bar{x}\| + \|\nabla h(\bar{x})-\nabla h(x_k)\|= o(\| x_k - \bar{x}\|).
\end{array}
\end{align*}
Then we have $\|  x_k + d_k -\bar{x}\|=o(\|x_k-\bar{x}\|)$ and deduce therefore from \cite[Proposition~8.3.18]{MR1955649} and Lemma~\ref{lemma1}(i) that
\begin{align}\label{eqSLC01}
\varphi(x_k + d_k) \leq \varphi(x_k) + \sigma \langle \nabla \varphi(x_k), d_k\rangle.
\end{align}
It follows from \eqref{eqSLC01} that $x_{k+1}=x_k + d_k$ if $k$ for large $k$. Applying \cite[Proposition~8.3.14]{MR1955649} yields the $Q$-superlinear convergence of $\{x_k\}$ to $\bar{x}$ as $k\to\infty$.

Finally, conditions \eqref{eq:1}--\eqref{eq:3} and the Lipschitz continuity of $\nabla \varphi$ around~$\bar{x}$ ensure the existence of $L>0$ such that
\begin{align*}
\frac{\eta}{\kappa^2}\| x_k - \bar{x} \|^2 &  \le\varphi(x_k) -\varphi(\bar{x}) \leq  r \| x_k - \bar{x} \|^2, \\
\quad\frac{1}{\kappa}\| x_k - \bar{x} \| & \leq   \| \nabla  \varphi(x_k)  \|\leq L\| x_k - \bar{x} \|
\end{align*}
for sufficiently large $k$, and therefore we get the estimates
\begin{equation}\label{estimationsCOR}
\begin{array}{ll}
\disp\frac{\varphi(x_{k+1}) -\varphi(\bar{x})  }{\varphi(x_k) -\varphi(\bar{x})} \leq \kappa r\disp\frac{ \| x_{k+1} - \bar{x} \|^2 }{ \| x_{k} -\bar{x}  \|^2},
\\
\quad\;\disp\frac{\| \nabla \varphi(x_{k+1}) \| }{ \| \nabla \varphi(x_{k})\|}\le\kappa L\disp\frac{ \| x_{k+1} - \bar{x} \| }{ \| x_{k} -\bar{x}\|},
\end{array}
\end{equation}
which thus conclude the proof of the theorem.
\end{proof}\vspace*{-0.25in}

\begin{remark}\label{rem:subregul}
The property of {\em strong metric subregularity} of {\em subgradient mappings}, which is a central assumption of Theorem~\ref{corSMR}, has been well investigated in variational analysis, characterized via second-order growth and coderivative type conditions, and applied to optimization-related problems; see, e.g., \cite{ag,dmn,MR3823783} and the references therein.
\end{remark}\vspace*{-0.05in}

The next theorem establishes the $Q$-superlinear and $Q$-quadratic convergence of the sequences generated by Algorithm~\ref{alg:1} provided that: $\xi>0$ (i.e., $\partial^2 g(x)$ is $\xi$-strongly positive-definite), $\rho_k=0$ for all $k\in \N$ (no regularization is used), $g$ is semismoothly differentiable at the cluster point $\bar{x}$, and the function $h$ can be expressed as the pointwise maximum of finitely many affine functions at $\bar{x}$, i.e., when there exist $(x^\ast_i, \alpha_i)_{i=1}^p \subseteq \mathbb{R}^n \times \R$ and $\epsilon >0$ such that 
\begin{align}\label{max_affine}
h(x)=\max_{i=1,\ldots, p} \left\{ \langle x^\ast_i,x\rangle +\alpha_i \right\}\; \text{ for all }\;x\in \mathbb{B}_\epsilon(\bar{x}).
\end{align}

\begin{theorem}\label{Cor:max_affine} In addition to the assumptions of Theorem~{\rm\ref{The01}}, suppose that $\xi>0$, $0<\zeta\leq\xi$, $\sigma \in (0,\frac{1}{2})$, $t_{\min}=1$, and $\rho_k=0$ for all $k\in\N$. Suppose also that the sequence $\{x_k\}$ generated by Algorithm~{\rm\ref{alg:1}} has an accumulation point $ \bar{x}$ at which $g$ is semismoothly differentiable and $h$ can be represented in form \eqref{max_affine}. Then we have the convergence $x_k\to\bar{x}$, $\varphi(x_k)\to\varphi(\bar{x})$, $w_k\to 0$, and $\nabla g(x_k)\to\nabla g(\bar{x})$ as $k\to\infty$ with at least $Q$-superlinear rate.  If in addition $g$ is of class $\mathcal{C}^{2,1}$ around $\ox$, then the rate of convergence is at least quadratic.
\end{theorem}\vspace*{-0.05in}
\begin{proof}
Observe that by \eqref{max_affine} and \cite[Proposition~1.113]{MR2191744} we have the inclusion
\begin{align}\label{FORMSUBD}
\partial (-h)(x)\subseteq \bigcup\big\{  -x^\ast_i\;\big|\;h(x)=\langle x^\ast_i,x\rangle +\alpha_i\big\}
\end{align}
for all $x$ near $\ox$. The rest of the proof is split into the five claims below.\vspace*{0.03in}

\noindent\textbf{Claim~1:} \emph{The sequence $\{x_k\}$ converges to $\bar{x}$ as $k\to\infty$.}\\
Observe that $\bar{x}$ is an isolated accumulation point. Indeed, suppose on the contrary that there is a sequence $\{y_\nu\}$ of accumulation points of $\{x_k\}$ such that $y_\nu \to \bar{x}$ as $\nu\to\infty$ with $y_\nu \neq \bar{x}$ for all $\nu\in\N$. Since each $y_\nu$ is accumulation point of $\{x_k\}$, they are stationary points of $\varphi$. The ${\cal C}^1$-smoothness of $g$ ensures that $\nabla g(y_\nu) \to \nabla g(\bar{x})$ as $\nu\to\infty$, and so \eqref{FORMSUBD} yields
$\nabla g(y_\nu)=x_{i_\nu}^\ast$ for large $\nu\in\N$. Since there are finitely many of $x_i^\ast$ in \eqref{max_affine}, we get that $\nabla g(y_\nu) = \nabla g(\bar{x})$ when $\nu$ is sufficiently large. Further, it follows from \cite[Theorem~5.16]{MR3823783} that the gradient mapping $\nabla g$ is strongly locally maximal monotone around  $\bar{x}$, i.e., there exist positive numbers $\epsilon$ and $r$ such that
\begin{align*}
\langle \nabla g(x)-\nabla g(y), x-y\geq r\| x -y\|^2\;\text{ for all } \;x,y\in\mathbb{B}_\epsilon(\bar{x}).
\end{align*}
Putting $x:=\bar{x}$ and $y:=y_\nu$ in the above inequality tells us that
$\bar{x}= y_\nu$ for large $\nu \in \N$, which is a contradiction. Applying  finally Theorem~\ref{The01}(iv), we complete the proof of this claim.\vspace*{0.03in}

\noindent\textbf{Claim 2:}  \emph{The sequence $\{x_k\}$ converges to $\bar{x}$ as $k\to\infty$ at least $Q$-superlinearly.}\\
As $x_k\to\bar{x}$, we have by Theorem~\ref{The01}(ii) that $w_k-\nabla g(x_k)\to-\nabla g(\bar{x})$, and so it follows from \eqref{FORMSUBD} that
there exists $i\in\{1,\ldots,p\}$ such that $h(\bar x ) = \langle x^\ast_i , \bar x\rangle +\alpha_i$, $h(x_k) = \langle x^\ast_i , x_k\rangle -\alpha_i$ and
$w_k-\nabla g(x_k)=-\nabla g(\bar{x})=-x_i^\ast$ for all $k$ sufficiently large. Define the auxiliary function $\hat{\varphi}:\mathbb{R}^n \to \Rex$ by 
\begin{align}\label{aux:func}
\hat{\varphi} (x):= g(x) -\langle x^\ast_i , x\rangle -\alpha_i
\end{align}
and observe that $\hat\varphi$ is $\mathcal{C}^{1,1}$ around $\ox$ and semismoothly differentiable at this point. We have the equalities
\begin{align}\label{eq_auxvarphi}
\varphi (x_k) =\hat{\varphi} (x_k),\;\varphi (\bar{x}) =\hat{\varphi} (\bar{x}),\; \nabla \hat{\varphi} (x_k)=w_k,\;\text{ and }\; \nabla \hat{\varphi} (\bar{x})=0
\end{align} 
for large $k$. It follows from $\partial^2\hat{\varphi} (x) = \partial^2  g(x)$ that the mapping $\partial^2\hat{\varphi}(\bar{x})+\rho_kI$ is $(\xi+\rho_k)$-lower-definite. Using \eqref{sublineal} and \eqref{inclusion01} with the replacement of $g$ by $\hat\varphi$ and taking \eqref{eq_auxvarphi} into account ensures the existence of $z_k \in \partial^2 \hat{\varphi}(x_k) (-x_k +  \bar{x}) +\rho_k(-x_k +  \bar{x})$ satisfying the estimate 
 \begin{align*}
 \| x_k + d_k -\bar{x}\| \leq \frac{1}{\xi+\rho_k} \|  \nabla \hat{\varphi}(x_k)- \nabla\hat{\varphi}(\bar{x})+ z_k\|.
 \end{align*}
 The triangle inequality and the semismoothness of $\nabla\hat \varphi$ at $\bar{x}$ yield
\begin{align*}
\|\nabla \hat{\varphi}(x_k)- \nabla \hat{\varphi}(\bar{x})+ z_k\|&\leq \| \nabla \hat{\varphi}(x_k)- \nabla \hat{\varphi}(\bar{x})+ z_k+\rho_k(x_k-\bar{x})\|+\rho_k\|x_k-\bar{x}\|\\
&=o(\|x_k-\bar{x}\|),
\end{align*}
which tells us that $\| x_k + d_k -\bar{x}\|=o(\| x_k - \bar{x}\|)$. Then it follows from\cite[Proposition~8.3.18]{MR1955649} and Lemma~\ref{lemma1}(i) above that
\begin{align}\label{eqSLC}
\hat{\varphi}(x_k + d_k) \leq \hat{\varphi}(x_k) + \sigma \langle \nabla\hat{\varphi}(x_k),d_k\rangle
\end{align}
whenever $k$ is sufficiently large. Applying finally \cite[Proposition~8.3.14]{MR1955649} verifies the claimed $Q$-superlinear convergence of $\{x_k\}$ to $\bar{x}$.\vspace*{0.03in}

\noindent\textbf{Claim~3:} \emph{The gradient mapping of  $\hat \varphi$ from \eqref{aux:func} is strongly metrically regular around $(\bar{x},0)$ and hence strongly metrically subregular at this point.}\\
Using the $\xi$-lower-definiteness of $\partial^2 \hat\varphi (\bar{x})$ and the pointbased coderivative characterization of strong local maximal monotonicity given in \cite[Theorem~5.16]{MR3823783}, we verify this property for $\nabla \hat \varphi$
around $\bar{x}$. Then \cite[Corollary~5.15]{MR3823783} ensures that $\nabla \hat \varphi$ is strongly metrically regular around $(\bar{x},0)$. \vspace*{0.03in}

\noindent\textbf{Claim~4:}  \emph{The sequences $\{ \varphi(x_k)\}$, $\{w_k\}$, and $\{\nabla g(x_k)\}$ converge at least Q-superlinearly to $\varphi(\bar{x})$, $0$, and $\nabla g(\bar{x})$, respectively.}\\
It follows from the estimates in \eqref{estimationsCOR}, with the replacement of $\ph$ by $\hat\varphi$ and with taking into account that
$\hat\varphi(x_k) -  \hat\varphi(\bar x) =\varphi(x_k)-\varphi(\bar x)$ and $ \nabla \hat \varphi(x_k) =  w_k$ due to \eqref{eq_auxvarphi}, that there exist constants $\alpha_1 , \alpha_2 >0$ such that
\begin{align*}
\frac{\varphi(x_{k+1}) -\varphi(\bar{x})  }{ \varphi(x_k) -\varphi(\bar{x})} &\leq \alpha_1 \frac{ \| x_{k+1} - \bar{x} \|^2 }{ \| x_{k} -\bar{x}  \|^2} \\
\frac{ \|w_{k+1 }\|  }{ \|w_k\|} &\leq  \alpha_2 \frac{ \| x_{k+1} - \bar{x} \| }{ \| x_{k} - \bar{x}  \|}
\end{align*}
provided that $k$ is sufficiently large. Recalling that $w_k-\nabla g(x_k)=-\nabla g(\bar{x})$ for large $k$ completes the proof of the claim. \vspace*{0.03in}

\noindent\textbf{Claim~5:} \emph{If $g$ is of class $\mathcal{C}^{2,1}$ around $\ox$, then the rate of convergence of the sequences above is at least quadratic.}\\
 It is easy to see that the assumed $\mathcal{C}^{2,1}$ property of $g$ yields this property of $\hat{\varphi}$ around $\ox$. Using estimate 
 \eqref{eqSLC}, we deduce this claim from the quadratic convergence of the classical Newton method; see, e.g., \cite[Theorem~5.18]{Aragon2019} and \cite[Theorem~2.15]{MR3289054}. This therefore completes the proof of the theorem.
\end{proof}\vspace*{-0.25in}

\begin{remark}\label{rem:long} Concerning Theorem~\ref{Cor:max_affine}, observe the following:

(i) It is important to emphasize that the performance of Algorithm~\ref{alg:1} revealed in Theorem~\ref{Cor:max_affine} is mainly due to the usage of the basic subdifferential of the function $-h$ in contrast to that of $h$, which is calculated as
\begin{equation}\label{h-sub}
\partial h(x)= \text{co} \left(\bigcup \left\{  x^\ast_i\;\bigg|\;h(x)=\langle x^\ast_i,x\rangle +\alpha_i \right\}\right)
\end{equation}
by \cite[Theorem~3.46]{MR2191744}. We can see from the proof of Theorem~\ref{Cor:max_affine} that it fails if the evaluation of $\partial(-h)(x)$ in \eqref{FORMSUBD} is replaced by the one of $\partial h(x)$ in \eqref{h-sub}.

(ii) The main assumptions of Theorem~\ref{Cor:max_affine} do not imply the smoothness of $\varphi$  at stationary points. For instance,  consider the nonconvex function   $\varphi: \mathbb{R}^n \to \R$ defined as in Example~\ref{example3.10} but letting now $h_i(x_i):= |x_i| +| 1- x_i|$.
 The  function $\varphi$ satisfies the assumptions of Theorem~\ref{Cor:max_affine} at any of its  stationary points $\{-2,0,2\}^n$, but $\varphi$ is not differentiable at $\bar{x}=0$; see Figure~\ref{example3.10reviplot2}.
	
 \begin{figure}[h!!]
\centering	\begin{tikzpicture}
\begin{axis}[
axis x line=center,
axis y line=center,
xtick={-5,-4,...,5},
ytick={-3,-2,...,2},
xlabel={$x$},
ylabel={$y$},
xlabel style={below right},
ylabel style={above left},
xmin = -6.5, xmax = 6.5,
ymin = -3.5, ymax = 2.5]
\addplot[
domain = -5.5:5.5,
samples = 200,
smooth,
ultra thick,
color=C0,
] {0.5*x*x  - abs(x)- abs(  1-x)};
\node at (0,-1)[circle,fill=red,inner sep=1.5pt]{};
\end{axis}
\end{tikzpicture}
\caption{ Plot of function  $\varphi_i(x)=\frac{1}{2}x^2 -|x| -| 1- x|  $ in Remark~\ref{rem:long}}
\label{example3.10reviplot2}
\end{figure}
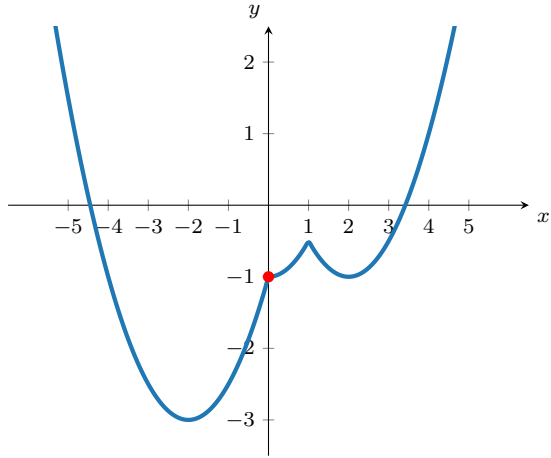
(iii) The functions $\varphi$, $g$, and $h$ in Example~\ref{example3.10} satisfy the assumptions of Theorem~\ref{Cor:max_affine}. Therefore, the convergence of the sequences generated by Algorithm~\ref{alg:1} is at least quadratic.
\end{remark}\vspace*{-0.33in}

\section{Convergence Rates under the Kurdyka--{\L}ojasiewicz Property}\label{sec:4}\vspace*{-0.1in} 

In this section, we verify the global convergence of Algorithm~\ref{alg:1} and establish convergence rates in the general setting of Theorem~\ref{The01} without additional assumptions of Theorems~\ref{corSMR} and \ref{Cor:max_affine} while supposing instead that the cost function $\ph$ satisfies the Kurdyka--{\L}ojasiewicz property. Recall that the \emph{Kurdyka--{\L}ojasiewicz property} holds for $\ph$ at $ \bar{x}$ if there exist $\eta >0$ and a continuous concave function $ \psi:[0,\eta] \to [0,\infty)$ with $\psi (0)=0$  such that $\psi$ is $\mathcal{C}^1$-smooth on $(0,\eta)$ with the strictly positive derivative $\psi'$ and that
\begin{align}\label{Kur-Loj}
\psi'\big(\varphi(x) - \varphi(\bar{x})\big)\,{\rm dist}\big(0;\partial \varphi(x)\big)\geq 1
\end{align}
for all $x\in \mathbb{B}_\eta(\bar{x})$ with $\varphi(\bar{x}) < \varphi(x) <\varphi( \bar{x} ) + \eta$, where ${\rm dist}(\cdot;\Omega)$ stands for the distance function of a set $\Omega$.

The first theorem of this section establishes the {\em global convergence} of iterative sequence generated by Algorithm~\ref{alg:1} to a {\em stationary point} of \eqref{EQ01}.\vspace*{-0.05in}  

\begin{theorem}\label{Teo:Kur-Loj}
In addition to the assumptions of Theorem~{\rm\ref{The01}}, suppose that the iterative sequence $\{x_k\}$ generated by Algorithm~{\rm\ref{alg:1}} has an accumulation point $\ox$ at which the Kurdyka--{\L}ojasiewicz property \eqref{Kur-Loj} is satisfied. Then $\{x_k\}$ converges $\ox$ as $k\to\infty$, which is a stationary point of problem~\eqref{EQ01}.
\end{theorem}\vspace*{-0.05in}
\begin{proof}
If Algorithm~\ref{alg:1} stops after a finite number of iterations, there is nothing to prove. Due to the decreasing property of $\{\ph(x_k)\}$ from Theorem~\ref{The01}(i), we can assume that $\varphi(x_k) > \varphi(x_{k+1}) $ for all $k \in \N$. Let $ \bar{x}$ be the accumulation point of $\{x_k\}$ where $\varphi$ satisfies the Kurdyka--{\L}ojasiewicz  inequality \eqref{Kur-Loj}, which by Theorem \ref{The01} is a stationary point of problem~\eqref{EQ01}. Since $\varphi$ is continuous, we have that $\varphi(\bar{x} )= \inf_{k\in\N}\varphi(x_k)$. Taking the constant $\eta>0$ and the function $\psi$ from \eqref{Kur-Loj} and remembering that $g$ is of class $\mathcal{C}^{1,1}$ around $\bar{x}$, suppose without loss of generality that $\nabla g$ is Lipschitz continuous on $\mathbb{B}_{2\eta}( \bar{x})$ with modulus $\kappa$. Let $k_0 \in \N$ be such that $x_{k_0} \in \mathbb{B}_{\eta/2}(\bar{x})$ and that
\begin{align}\label{iq:01Ku-Loj}
\varphi (\bar{x}) < \varphi(x_{k})  < \varphi(\bar{x}) +\eta, \quad\frac{ \kappa +\rho_{\max} }{\sigma \zeta }\psi\big( \varphi (x_{k}-\varphi(\bar{x})\big) < \eta/2
\end{align}
 for all $k \geq k_0$, where $\sigma \in  (0,1)$, $\zeta>0$, and $\rho_{\max}>0$ are the constants of Algorithm~\ref{alg:1}. The rest of the proof is split into the following three steps.\vspace*{0.03in}

\noindent\textbf{Claim~1:} \emph{Let $k \geq k_0 $ be such that     $x_k \in \mathbb{B}_\eta (\bar{x})$. Then we have the estimate}
\begin{align}\label{eq:Kur-Loj}
\| x_k -x_{k+1}\| \leq \frac{ \kappa+\rho_k }{\sigma\zeta }\big( \psi( \varphi(x_{k})  - \varphi(\bar{x})\big) -  \psi\big( \varphi(x_{k+1})  - \varphi(\bar{x})\big)\big).
\end{align}
Indeed, it follows from \eqref{coder:sub}, \eqref{EQALG01}, and  \cite[Theorem~1.22]{MR3823783} that
\begin{equation}\label{dist:inq}
\begin{array}{ll}
{\rm dist}(0;\partial \varphi (x_k)\big) &\leq \| w_k \| \leq \| w_k +\rho_kd_k\|+\rho_k\|d_k\|\\
&\leq (\kappa+\rho_k)  \| d_k \|  = \disp\frac{ \kappa +\rho_k}{ \tau_k } \| x_{k+1}  - x_k\|.
\end{array}
\end{equation}
Then using Step~5 of Algorithm~\ref{alg:1},  Lemma~\ref{lemma1}, the Kurdyka--{\L}ojasiewicz inequality \eqref{Kur-Loj}, the concavity of $\psi$, and  estimate \eqref{dist:inq} gives us 
\begin{align*}
\| x_k - &x_{k+1}\|^2  =  \tau^2_k  \| d_k\|^2 \leq  \frac{     \tau_k  }{\sigma\zeta }\big(  \varphi(x_k ) - \varphi(x_{k+1})\big) \\
& \leq      \frac{     \tau_k  }{\sigma\zeta }{\rm dist}\big(0;\partial \varphi (x_k)\big)\,\psi'\big(\varphi(x_k) - \varphi(\bar{x})\big)  	  \big(  \varphi(x_k ) - \varphi(x_{k+1})\big)\\
&  \leq \frac{     \tau_k  }{\sigma\zeta }{\rm dist}\big(0;\partial \varphi (x_k)\big)\big(\psi( \varphi(x_k) - \varphi(\bar{x})\big)  - \psi(  \varphi(x_{k+1}) - \varphi(\bar{x})\big)\big) \\
&\leq \frac{ \kappa+\rho_k }{\sigma\zeta } \| x_{k+1}  - x_k\|\big(     \psi\big( \varphi(x_k) - \varphi(\bar{x})\big)  - \psi\big(\varphi(x_{k+1}) - \varphi(\bar{x})\big)\big),
 \end{align*}
which therefore verifies the claimed inequality \eqref{eq:Kur-Loj}.\vspace*{0.03in}

\noindent\textbf{Claim~2:} \emph{For every $k \geq k_0 $, we have the inclusion  $x_{k} \in \mathbb{B}_{\eta} (\bar{x})$.}\\
Suppose on the contrary that there exists $k> k_0$ with $x_k \notin \mathbb{B}_\eta (\bar{x})$ and define $\bar{k}:=\min\left\{ k> k_o\;\big|\; x_k \notin  \mathbb{B}_\eta (\bar{x})  \right\}$. Since for $k\in\{k_0,\ldots,\bar{k}-1\}$ the  estimate in \eqref{eq:Kur-Loj} is satisfied, we get by using~\eqref{iq:01Ku-Loj} that
\begin{align*}
\| x_{ \bar{k} } - \bar{x} \|&  \leq \| x_{k_0} -  \bar{x}\|  +  \sum_{k=k_0}^{\bar{k}-1}\| x_{k}  - x_{k+1}\| \\
&\leq  \| x_{k_0} -  \bar{x}\| + \frac{\kappa+\rho_{\max}}{\sigma\zeta } \sum_{k=k_0}^{\bar{k}-1}\big( \psi\big( \varphi(x_{k})  - \varphi(\bar{x})\big) - \psi\big( \varphi(x_{k+1})  - \varphi(\bar{x}) \big)\big)\\
&\leq  \| x_{k_0} -  \bar{x}\| + \frac{\kappa+\rho_{\max}}{\sigma\zeta }   \psi\big( \varphi(x_{k_0})  - \varphi(\bar{x})\big) \leq \eta,
\end{align*}
which contradicts our assumption and  thus verifies this claim.\vspace*{0.03in}

\noindent\textbf{Claim~3:}  \emph{We have that $\sum_{k=1}^{\infty} \| x_k - x_{k+1}\| < \infty$, and consequently the sequence $\{x_k\}$ converges to $\bar{x}$ as $k\to\infty$.}\\
It follows from Claim~1 and Claim~2 that \eqref{eq:Kur-Loj} holds for all $k \geq k_{0}$. Thus
\begin{align*}
\sum_{k=1}^{\infty} \| x_k - x_{k+1}\|& \leq \sum_{k=1}^{k_0-1} \| x_k - x_{k+1}\| + \sum_{k=k_0}^{\infty} \| x_k - x_{k+1}\| \\
& \leq   \sum_{k=1}^{k_0-1} \| x_k - x_{k+1}\| +   \frac{\kappa+\rho_{\max}}{\sigma\zeta } \psi\big( \varphi(x_{k_0})  - \varphi(\bar{x})\big) <\infty,
\end{align*}
which therefore completes the proof of the theorem.
\end{proof}\vspace*{-0.1in}
 
 The next theorem establishes {\em convergence rates} for iterative sequence $\{x_k\}$ in Algorithm~\ref{alg:1} provided that the function $\psi$ in \eqref{Kur-Loj} is selected in a special way. Since the proof while using Theorem~\ref{Teo:Kur-Loj}, is similar to the corresponding one from \cite[Theorem~4.9]{MR4078808} in a different setting, it is omitted. \vspace*{-0.05in}

 \begin{theorem}\label{COR:Kur-Loj}
 In addition to the assumptions of Theorem~{\rm\ref{Teo:Kur-Loj}}, suppose that the Kurdyka--{\L}ojasiewicz property \eqref{Kur-Loj} holds at the accumulation point $\bar{x}$ with $\psi(t):= M t^{1-\theta}$ for some $M>0$ and $\theta\in[0,1)$. The following assertions hold:

 {\bf(i)} If $\theta =0$, then the sequence $\{x_k\}$ converges in a finite number of steps.
  
 {\bf(ii)} If $\theta\in (0,1/2]$,  then the sequence $\{x_k\}$ converges at least linearly.
 
{\bf(iii)} If $\theta \in (1/2,1)$, then there exist $\mu >0$ and $k_0\in \N$ such that
 \begin{align*}
 \|x_ k  - \bar{x}\| \leq \mu k^{-\frac{1-\theta }{ 2\theta -1 } }\;\text{ for all }\;k \geq k_0.
 \end{align*}
\end{theorem}\vspace*{-0.1in}

\begin{remark}
Together with our main Algorithm~\ref{alg:1}, we can consider its modification with the replacement of $\partial(-h)(x_k)$ by $-\partial h(x_k)$. In this case, the most appropriate version of the Kurdyka--{\L}ojasiewicz inequality \eqref{Kur-Loj}, ensuring the fulfillment the corresponding versions of Theorem~\ref{Teo:Kur-Loj} and \ref{COR:Kur-Loj}, is the one
 \begin{align*}
 \psi\big(\varphi(x)-\varphi(\bar{x})\big)\,{\rm dist}\big(0;\partial^0\varphi(x)\big)\geq 1
 \end{align*}
expressed in terms of the {\em symmetric subdifferential} $\partial^0\varphi(x)$ from \eqref{sym}. Note that the latter is surely satisfied where the symmetric subdifferential is replaced by the {\em generalized gradient} 
$\overline{\partial}\varphi(x)$, which is the convex hull of $\partial^0\varphi(x)$. 
\end{remark}\vspace*{-0.35in}

 \section{Applications to Structured Constrained Optimization}\label{sec:5}\vspace*{-0.05in}

In this section, we present implementations and specifications of our main RCSN Algorithm~\ref{alg:1} for two structured classes of optimization problems. The first class contains functions represented as sums of two nonconvex functions one of which is smooth, while the other is extended-real-valued. The second class concerns minimization of smooth functions over closed constraint sets.\vspace*{-0.25in}

\subsection{Minimization of Structured Sums}\label{subsec1}\vspace*{-0.05in}

Here we consider the following class of structured optimization problems:
\begin{equation}\label{ProFBE}
\min_{x\in \R^n}\varphi(x):=f(x)+\psi(x),
\end{equation}
where $f:\mathbb{R}^n\to\mathbb{R}$ is of class $\mathcal{C}^{2,1}$ with the $L_f$-Lipschitzian gradient, and where $\psi: \mathbb{R}^n \to \Rex$ is an extended-real-valued prox-bounded function with the threshold $\lambda_\psi>0$. When both functions $f$ and $\psi$ are convex, problems of type \eqref{ProFBE} have been largely studied under the name of ``convex composite optimization" emphasizing the fact that $f$ and $\psi$ are of completely different structures. In our case, we do not impose any convexity of $f,\psi$ and prefer to label \eqref{ProFBE} as {\em minimization of structured sums} to avoid any confusions with optimization of function compositions, which are typically used in major models of variational analysis and constrained optimization; see, e.g., \cite{MR1491362}. 

In contrast to the original class of {\em unconstrained}  problems of {\em difference programming} \eqref{EQ01}, the structured {\em sum optimization} form \eqref{ProFBE} covers optimization problems with {\em constraints} given by $x\in\dom\psi$. Nevertheless, we show in what follows that the general class of problem \eqref{ProFBE} can be reduced under the assumptions imposed above to the difference form \eqref{EQ01} satisfying the conditions for the required performance of Algorithm~\ref{alg:1}. 

This is done by using an extended notion of envelopes introduced by Patrinos and Bemporad in \cite{Patrinos2013}, which is now commonly referred as the \emph{forward-backward envelope}; see, e.g., \cite{MR3845278}.\vspace*{-0.1in}

\begin{definition}
Given $ \varphi = f + \psi$ and $\lambda >0$, the \emph{forward-backward envelope} (FBE) of the function $\varphi$ with the parameter $\lambda$ is defined by
\begin{align}\label{FBE}
\varphi_\lambda(x) :=\inf_{z \in \mathbb{R}^n}\Big\{ f(x) + \langle \nabla f(x),z-x\rangle + \psi(z) +\frac{1}{2\lambda }\| z- x\|^2\Big\}.
\end{align}
\end{definition}

Remembering the constructions of the Moreau envelope \eqref{moreau} and the Asplund function \eqref{asp} allows us to represent $\varphi_\lambda$ for every $\lambda \in (0, \lambda_\psi)$ as:
\begin{equation}\label{rep02}
\begin{array}{ll}
\varphi_\lambda(x)&= f(x) -\disp\frac{\lambda}{2}\| \nabla f(x)\|^2 +  \MoreauYosida{\psi}{\lambda}\big(x-\lambda \nabla f(x)\big) \\
&=\disp f(x) +\frac{1}{2\lambda} \|x\|^2 -\langle \nabla f(x), x\rangle  - \Asp{\lambda}{\psi}\big(x- \lambda \nabla f(x)\big).
 \end{array}
 \end{equation}

\begin{remark}\label{rem:infimum}
It is not difficult to show that whenever $\nabla f$ is $L_f$-Lipschitz on $\mathbb{R}^n$ and $\lambda \in (0,\frac{1}{L_f})$, the optimal values in problems \eqref{EQ01} and \eqref{FBE} are the same
\begin{align}\label{eqinf}
\inf_{x\in \mathbb{R}^n}\varphi_\lambda(x)= \inf_{x\in \mathbb{R}^n } \varphi(x).
\end{align}
Indeed, the inequality ``$\leq $'' in \eqref{eqinf} follows directly from the definition of $\varphi_\lambda$. The reverse inequality in \eqref{eqinf} is obtained by
\begin{equation*}
\begin{array}{ll}
\disp\inf_{x\in \mathbb{R}^n} \varphi_\lambda(x)=\disp\inf_{x\in \mathbb{R}^n}\disp\inf_{ z \in \R^n } \Big\{ f(x) + \langle \nabla f(x),z-x\rangle + \psi(z) +\disp\frac{1}{2\lambda  }\|z- x\|^2\Big\}\\
\geq\disp\inf_{x\in \mathbb{R}^n}\disp\inf_{ z \in \R^n}\Big\{ f(z) -\frac{L_f}{2} \|z-x\|^2 + \psi(z) +\disp\frac{1}{2\lambda  }\| z- x\|^2\Big\}\\
= \disp\inf_{ z \in \R^n }\disp\inf_{x\in \mathbb{R}^n}\Big\{ f(z) + \psi(z)+\Big(\frac{1}{2\lambda } -\disp\frac{L_f}{2}\Big)\| z- x\|^2\Big\}= \disp\inf_{ z \in \R^n }\varphi(x).
\end{array}
\end{equation*}
Moreover, \eqref{eqinf} does not hold if $\nabla f$ is not Lipschitz continuous on $\R^n$. Indeed, consider $f(x):=\frac{1}{4} x^4$ and $\psi:=0$. Then we have $\inf_{x\in\mathbb{R}^n} \varphi(x) =0$ while $\varphi_\lambda(x)=\frac{1}{4} x^4 - \frac{\lambda}{2} x^6 $, which yields $\inf_{x\in \mathbb{R}^n} \varphi_\lambda(x)=-\infty$, and so \eqref{eqinf} fails.
\end{remark}\vspace*{-0.05in}
	
The next theorem shows that FBE  \eqref{FBE} can be written as the difference of a $\mathcal{C}^{1,1}$ function and a Lipschitzian prox-regular function. Furthermore, it establishes relationships between minimizers and critical points of $\varphi$ and $\varphi_\lambda$.\vspace*{-0.1in}

\begin{theorem}\label{diff_repr_varphi}
Let  $\varphi=f+\psi$, where $f$ is of class $\mathcal{C}^{2,1}$ and where $\psi$ is prox-bounded with threshold $\lambda_\psi>0$. Then for any $\lambda \in (0,\lm_\psi)$, we have the inclusion
\begin{equation}\label{Subbasic}
\partial\varphi_\lambda (x)\subseteq \lambda^{-1}\big( I- \lambda \nabla^2 f(x)\big) \big(x -\Prox{\psi}{\lambda}\big(x-\lambda \nabla f(x)\big)\big). 
\end{equation}
Furthermore, the following assertions are satisfied:

{\bf(i)} \label{diff_repr_varphi_a} If $x\in \mathbb{R}^n$ is a stationary point of $\varphi_\lambda$, then  $0\in \Hat{\partial} \varphi(x)$ provided that the matrix $ I- \lambda \nabla^2 f(x)$ is nonsingular.

{\bf(ii)} \label{diff_repr_varphi_c} The FBE \eqref{FBE} can be written as $\varphi_\lambda=g-h$, where $g(x):= f(x) +\frac{1}{2\lambda} \|x\|^2 $ is of class $\mathcal{C}^{2,1}$, and where $h(x):= \langle \nabla f(x),x\rangle + \Asp{\lambda}{\psi}(x- \lambda \nabla f(x))$ is locally Lipschitzian and prox-regular on $\R^n$. Moreover,   $\nabla^2 g(x)$ is  $\xi$-lower-definite for all $x\in \mathbb{R}^n$ with $\xi:=\frac{1}{\lambda} - L_f$.

{\bf(iii)} \label{diff_repr_varphi_d} If $\psi:=\delta_{C}$ for a closed set $C$, then $\partial (-\Asp{\lambda}{\psi})=-\frac{1}{\lambda}\mathtt{P}_C$, where $\mathtt{P}_C$ denotes the $($generally set-valued$)$ projection operator onto $C$. In this case, inclusion \eqref{Subbasic} holds as an equality.

{\bf(iv)} \label{diff_repr_varphi_b}  If both $f$ and $\psi$ are convex, we have that $\varphi_\lambda = g - h$, where $g(x):= f(x) + \MoreauYosida{\psi}{\lambda} (x-\lambda \nabla f(x))$ and $h(x):= \frac{\lambda}{2}\| \nabla f(x)\|^2$ are of class $\mathcal{C}^{1,1}$ $($and hence prox-regular$)$ on $\R^n$, and that
\begin{align}\label{eqconvexcase}
\big\{x \in \mathbb{R}^n\;\big|\;\nabla \varphi_\lambda(x)=0\big\} =\big\{ x\in \mathbb{R}^n\;\big|\;0 \in \partial \varphi(x)\big\}
\end{align}
provided that $ I- \lambda \nabla^2 f(x)$ is nonsingular at any stationary point of $\varphi_\lambda$.
\end{theorem}\vspace*{-0.05in}
\begin{proof}
Observe that inclusion \eqref{Subbasic} follows directly by applying the basic subdifferential sum and chain rules from \cite[Theorem~2.19 and Corollary~4.6]{MR3823783}, respectively, the first representation of $\ph_\lm$ in \eqref{rep02} with taking into account the results of Lemma~\ref{Lemma5.1}. Now we pick any stationary point $x\in\R^n$ of the FBE $\ph_\lambda$ and then  deduce from
$0\in \partial \varphi_\lambda(x)$ and \eqref{Subbasic} that 
\begin{equation*}
x \in\Prox{\psi}{\lambda}\big(x-\lambda \nabla f(x))\big),
\end{equation*}
which readily implies that $0 \in \nabla f(x) +\hat{\partial} \psi(x)=\hat{\partial} \varphi(x)$ and thus verifies (i). Assertion (ii) follows directly from  Proposition~\ref{Lemma5.1} and the smoothness of $f$.

To prove (iii), we need to verify the reverse inclusion ``$\supseteq$'' in \eqref{eq_sub_eq01}, for which it suffices to show that the inclusion $v\in\mathtt{P}_C(x)$ yields $v\not\in{\rm co}(\mathtt{P}_C(x)\setminus\{v\})$. On the contrary, if $v\in\mathtt{P}_C(x)\cap{\rm co}(\mathtt{P}_C(x)\setminus\{v\})$, then there exist $c_1,\ldots,c_m\in P_C(x)\setminus\{v\}$ and $\mu_1,\ldots,\mu_m\in (0,1)$ such that $v=\sum_{i=1}^m\mu_i c_i$ with $\sum_{i=1}^m \mu_i=1$. By definition of the projection, we get the equalities
\begin{equation*}
|c_1-x\|^2=\ldots=\|c_m-x\|^2=\|v-x\|^2=\Big\|\sum_{i=1}^m\mu_i(c_i-x)\Big\|^2,
\end{equation*}
which contradict the strict convexity of $\|\cdot\|^2$ and thus verifies (iii).

The first statement in (iv) follows from the differentiability of $f$ and of the Moreau envelope $\MoreauYosida{\psi}{\lambda}$ by \cite[Theorem~2.26]{MR1491362}. Further, the inclusion ``$\subseteq $'' in \eqref{eqconvexcase} is a consequence of (i). To justify the reverse inclusion in \eqref{eqconvexcase}, observe that any $x$ satisfying $0\in\partial\ph(x)$ is a global minimizer of the convex function $\ph$, and so $x = \Prox{\psi}{\lambda} (x-\lambda \nabla f(x))$. The differentiability of $\varphi_\lambda$ and \eqref{Subbasic} (which holds as an equality in this case) tells us that $\nabla\varphi_\lambda(x)=0$, and thus \eqref{eqconvexcase} holds. This completes the proof of the theorem.
\end{proof}\vspace*{-0.25in}

 \begin{remark}\label{rem:DC repres} Based on Theorem~\ref{diff_repr_varphi}(ii), it is not hard to show that the FBE function $\varphi_\lambda$ can be represented as a difference of convex functions. Indeed, since $\Asp{\lambda}{\psi}$ is a locally Lipschitzian and prox-regular function, we have by \cite[Corollary~3.12]{MR2101873} that $h$ is a lower-$\mathcal{C}^2$ function, and hence by \cite[Theorem~10.33]{MR1491362}, it is locally a DC function. Similarly, $g$ being a $\mathcal{C}^2$ function is a DC function, so the difference $\varphi=g-h$ is also a DC function. However, it is difficult to determine for numerical purposes what is an appropriate representation of $\varphi$ as a difference of convex functions. Moreover, such a representation of the objective in terms of convex functions may generate some theoretical and algorithmic challenges as demonstrated below in Example~\ref{Example5.4}.
 \end{remark}\vspace*{-0.35in}

\subsection{Nonconvex Optimization with Geometric Constraints}\label{subsec:5.1}\vspace*{-0.1in}

This subsection addresses the following problem of {\em constrained optimization} with explicit geometric constraints given by:
\begin{equation}\label{prob:constrained}
\mbox{minimize }\;f(x)\;\mbox{ subject to }\;x\in C,
\end{equation}
where $f:\mathbb{R}^n\to\mathbb{R}$ is of class $\mathcal{C}^{2,1}$, and where $C\subseteq\mathbb{R}^n$ is an arbitrary closed set. Due to the lack of convexity, most of the available algorithms in the literature are not able to directly handle this problem. Nevertheless, Theorem~\ref{diff_repr_varphi} provides an effective machinery allowing us to reduce \eqref{prob:constrained} to an optimization problem that can be solved by using our developments. Indeed, define $\psi(x): = \delta_{C}(x)$ and observe that $\psi$ is prox-regular with threshold $\lambda_\psi = \infty$.
In this setting, FBE \eqref{FBE} reduces to the formula
\begin{equation*}
\varphi_\lambda(x)=f(x)-\frac{\lambda}{2}\|\nabla f(x)\|^2+\frac{1}{2\lambda}{\rm dist}^2\big(x-\lambda\nabla f(x);C\big).
\end{equation*}
Furthermore, it follows from Theorem~\ref{diff_repr_varphi}(iii) that
 \begin{align*}
 \partial \varphi_\lambda(x)=\lambda^{-1}\big(I-\lambda \nabla^2 f(x)\big)\big(x -    \mathtt{P}_C\big( x -\lambda \nabla f(x)\big)\big).
 \end{align*}
Based on Theorem~\ref{diff_repr_varphi}, we deduce from Algorithm~\ref{alg:1} with $\rho_k=0$ its following version to solve the constrained problem \eqref{prob:constrained}.\vspace*{-0.2in}

\begin{algorithm}[ht!]
\begin{algorithmic}[1]
\Require{$x_0 \in \R^n$, $\beta \in (0,1)$, $t_{\min}>0 $ and $\sigma\in(0,1)$.}
\For{$k=0,1,\ldots$}
\State Take $w_k\in \left(\lambda^{-1}I-\nabla^2 f(x_k)\right)\big( x_k -  \mathtt{P}_C\big(x-\lambda \nabla f(x_k)\big)\big)$.  
\State If $  w_k=0$, STOP and return~$x_k$. Otherwise set $d_k$ as the solution to the linear system $(\nabla^2 f(x_k)+\lambda^{-1}I)d_k=w_k$.
\State Choose any $\overline{\tau}_k\geq t_{\min}$. Set $\overline{\tau}_k:=\tau_k$.
\While{$\varphi_\lambda(x_k + \tau_k d_k) > \varphi_\lambda(x_k) +\sigma \tau_k \langle \nabla w_k , d_k\rangle $}
\State $\tau_k = \beta \tau_k$.
\EndWhile
\State Set $x_{k+1}:=x_k + \tau_kd_k$. \label{step5_2}
\EndFor
\end{algorithmic}
\caption{Projected-like Newton algorithm for constrained optimization}\label{alg:3}
\end{algorithm}\vspace*{-0.2in}

To the best of our knowledge, Algorithm~\ref{alg:3} is new even for the case of convex constraint sets $C$. All the results obtained for Algorithm~\ref{alg:1} in Sections~\ref{sec:3} and \ref{sec:4} can be specified for Algorithm~\ref{alg:3} to solve problem \eqref{prob:constrained}. For brevity, we present just the following direct consequence of Theorem~\ref{The01}. \vspace*{-0.05in}

\begin{corollary}\label{Cor:Theo01}
Considering problem \eqref{prob:constrained}, suppose that $f: \mathbb{R}^n \to\mathbb{R}$ is of class $\mathcal{C}^{2,1}$, that $C\subset\R^n$ is closed, and that 
$\inf_{x\in C}f(x) >-\infty$.  Pick an initial point $x_0 \in \mathbb{R}^n$ and a parameter $\lambda\in (0, \frac{1}{L_f})$. Then Algorithm~{\rm\ref{alg:3}} either stops at a point $x$ such that $0\in\nabla f(x)+\Hat{N}_C(x)$, or generates infinite sequences $\{x_k\}$, $\{\varphi_\lambda(x_k)\}$, $\{w_k\}$, $\{d_k\}$, and $\{\tau_k\}$ satisfying the assertions:

{\bf(i)} \label{Cor:The01a} The sequence $\{\varphi_\lambda(x_k)\}$ monotonically decreases and converges.

{\bf(ii)} \label{Cor:The01b} If $\{x_{k_j}\}$ is a bounded subsequence of $\{x_k\}$, then $\inf_{j\in\N} \tau_{k_j}>0$ and  
\begin{align*}
\sum\limits_{j \in \N}  \| d_{k_j}\|^2 < \infty,\; \sum\limits_{j\in \N } \| x_{k_j +1} - x_{k_j}\|^2< \infty,\;\sum\limits_{j\in \N } \| w_{k_j}\|^2 < \infty.
\end{align*}
If, in particular, the entire sequence $\{x_k\}$ is bounded, then the set of its accumulation points is nonempty, closed, and connected.
  
{\bf(iii)} \label{Cor:The01c} If $x_{k_j} \to \bar{x}$ as $j\to\infty$, then $0\in\nabla f(\bar{x})+\Hat{N}_C(\bar{x})$ and the equality $\varphi_\lambda(\bar{x})=\inf_{k\in \N} \varphi_\lambda(x_k)$ holds.

{\bf(iv)} \label{Cor:The01d} If the sequence $\{x_k\}$ has an isolated accumulation point $\bar{x}$, then it converges to $\bar{x}$ as $k\to\infty$, and we have $0\in\nabla f(\bar{x})+\Hat{N}_C(\bar{x})$.
\end{corollary}\vspace*{-0.05in}

The next example illustrates our approach to solve \eqref{prob:constrained} via Algorithm~\ref{alg:3} in contrast to algorithms of the DC type.\vspace*{-0.05in}

 \begin{example}\label{Example5.4} 		
Consider the minimization of a quadratic function over a closed (possibly nonconvex) set $C$:
 \begin{align}\label{ProblemQUAD_example}
 \mbox{minimize }\;\frac{1}{2}x^\tr Q x + b^\tr x\;  \text{ subject to }\;
 x\in C,	
 \end{align}
 where $Q$ is a symmetric matrix, and where $b \in \mathbb{R}^n$. In this setting, FBE \eqref{FBE} can be written as $\varphi_\lambda(x) = g(x) - h(x)$  with
 \begin{equation}\label{func_h_Example55}
 \begin{array}{ll}
 g(x)&:=\disp\frac{1}{2}x^\tr\big(  Q + \lambda^{-1} I\big) x + b^\tr x,\\  
h(x)& :=  x^\tr  Q  x + b^\tr x+ \Asp{\lambda}{\psi}\big((I- \lambda Q)x-\lambda b\big).
\end{array}
\end{equation}
Our method {\em does not require} a DC decomposition of the objective function $\varphi_\lambda$. Indeed the function $h$ in \eqref{func_h_Example55} is generally nonconvex. Specifically, consider $Q=\begin{bsmallmatrix}0&-1\\-1&0\end{bsmallmatrix}$, $b=(0,0)^\tr$, and $C$ being the unit sphere centered at the origin. Then $g$ in \eqref{func_h_Example55} is strongly convex for any $\lambda \in (0,1)$, while $h$ therein is not convex whenever $\lambda >0$. More precisely, in this case we have
$$
h(x_1,x_2) = -2x_1x_2+ \Asp{\lambda}{\psi}(x_1+\lambda x_2,\lambda x_1+x_2)\;\mbox{ with}
$$
\begin{align*}
\Asp{\lambda}{\psi}(x)=  \frac{1}{2\lambda}\left(\|x\|^2-d_C^2(x)\right)=\frac{1}{2\lambda}\left(\|x\|^2-(\|x\|-1)^2\right)=\frac{1}{2\lambda}\left(2\|x\|-1\right).
\end{align*}
This tells us, in particular, that 
$$
h(-1/2,-1/2)-\frac{1}{2}h(-1,-1)-\frac{1}{2}h(0,0)=\frac{1}{2},
$$
and thus $h$ is not convex regardless of the value of $\lambda$; see Figure~3.\vspace*{-0.2in}

\begin{figure}[h!!]
\centering
\includegraphics[width=.7\textwidth]{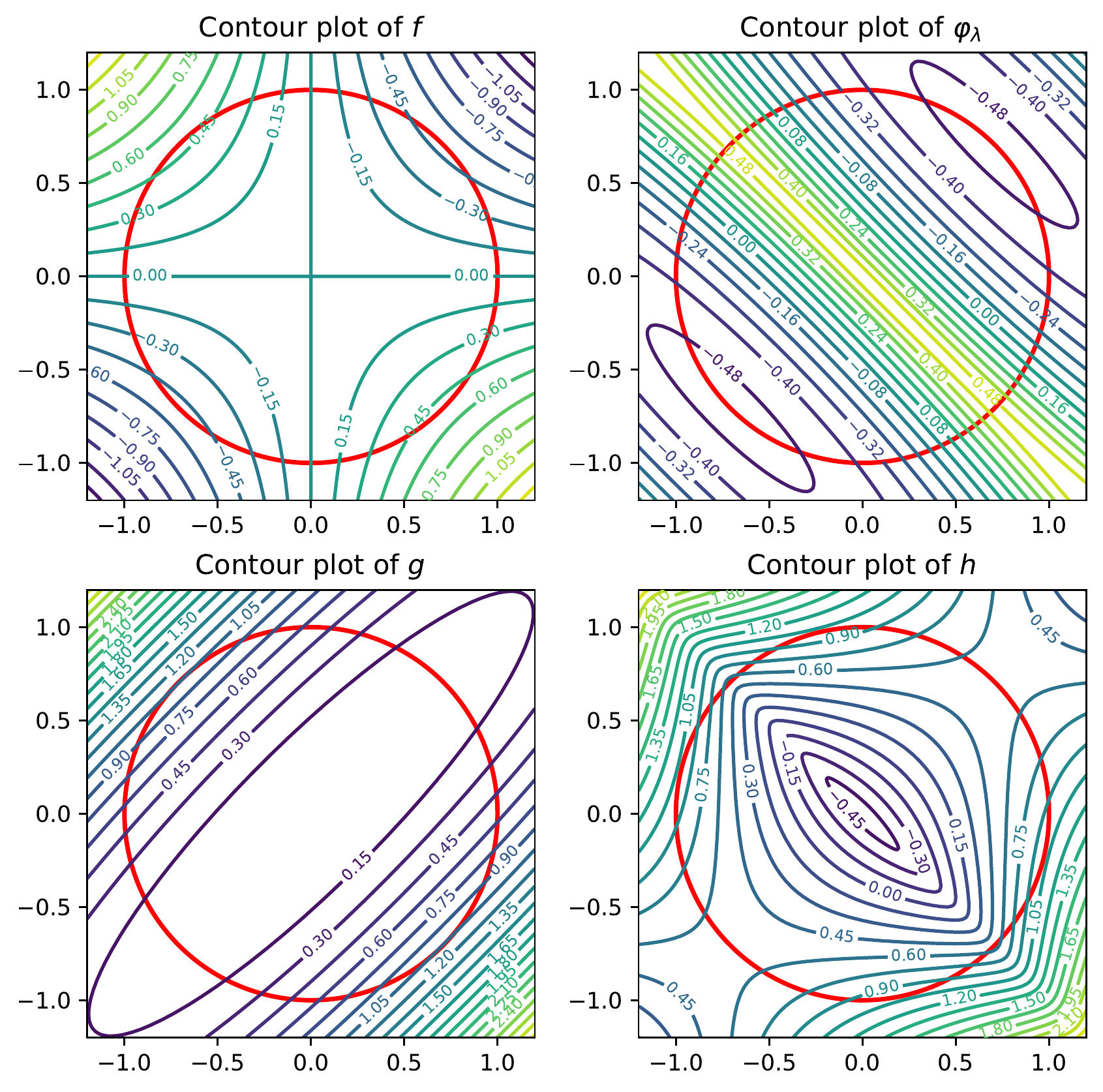}
\caption{Contour plot of  the functions $f$, $\varphi_\lambda$, $g$ and $h $ in \eqref{func_h_Example55} with $\lambda =0.9$}
\end{figure}
\end{example}\vspace*{-0.15in}

\section{Further Applications and Numerical Experiments}\label{sec:6}\vspace*{-0.1in}

In this section, we demonstrate the performance of Algorithm~\ref{alg:1} and Algorithm~\ref{alg:3} in two different problems. The first problem is smooth and arises from the study of system biochemical reactions. It can be successfully tackled with DCA-like algorithms, but they require to solve subproblems whose solutions cannot be analytically computed and are thus time-consuming. This is in contrast to Algorithm~\ref{alg:1}, which only requires solving the linear equation \eqref{EQALG01} at each iteration. The second problem is nonsmooth and consists of minimizing a quadratic function under both convex and nonconvex constraints. Employing FBE \eqref{FBE} and Theorem~\ref{diff_repr_varphi}, these two problems can be attacked by using DCA, BDCA, and Algorithm~\ref{alg:3}.

Both Algorithms~\ref{alg:1} and \ref{alg:3} have complete freedom in the choice of the initial value of the stepsizes $\overline{\tau}_k$ in Step~4, as long as they are bounded from below by a positive constant $t_{\min}$, while the choice of $\overline{\tau}_k$ totally determines the performance of the algorithms. On the one hand, a small value would permit the stepsize to get easily accepted in Step~5, but it would imply little progress in the iteration and (likely) in the reduction of the objective function, probably making it more prone to stagnate at local minima. On the other hand, we would expect a large value to ameliorate these issues, while it could result in a significant waste of time in the linesearch Steps~5-7 of both algorithms.

Therefore, it makes sense to consider a choice which sets the trial stepsize $\overline{\tau}_k$ depending on the stepsize $\tau_{k-1}$ accepted in the previous iteration, perhaps increasing it if no reduction of the stepsize was needed. This technique was introduced in~\cite[Section~5]{MR4078808} under the name of \emph{Self-adaptive trial stepsize}, and it was shown there that this accelerates the performance of BDCA in practice. A similar idea is behind the so-called \emph{two-way backtracking} linesearch, which was recently proposed in~\cite{Truong2021} for the gradient descent method, showing good numerical results on deep neural networks. In contrast to BDCA, our theoretical results require $t_{\min}$ to be strictly positive, so the technique should be slightly adapted as shown in Algorithm~\ref{alg:self-adaptive}. Similarly to \cite{MR4078808}, we adopt a conservative rule of only increasing the trial stepsize $\overline{\tau}_k$ when two consecutive trial stepsizes were accepted without decreasing them.\vspace*{0.1in}

\begin{algorithm}[ht!]
\begin{algorithmic}[1]
\Require{$\gamma>1$, $\overline{\tau}_0>0$}.
\State Obtain $\tau_0$ by Steps~5-7 of Algorithms~\ref{alg:1} or \ref{alg:3}.
\State Set $\overline{\tau}_1:=\max\{\tau_0,t_{\min}\}$ and obtain $\tau_1$ by Steps~5-7 of Algorithms~\ref{alg:1} or \ref{alg:3}.
\For{$k=2,3,\ldots$}
\If{$\tau_{k-2}=\overline{\tau}_{k-2}$ \textbf{and} $\tau_{k-1}=\overline{\tau}_{k-1}$}
\State $\overline{\tau}_{k}:=\gamma\tau_{k-1}$;
\Else
\State $\overline{\tau}_{k}:=\max\{\tau_{k-1},t_{\min}\}$.
\EndIf
\State Obtain $\tau_k$ by Steps~5-7 of Algorithms~\ref{alg:1} or \ref{alg:3}.
\EndFor
\end{algorithmic}
\caption{Self-adaptive trial stepsize}\label{alg:self-adaptive}
\end{algorithm}
\vspace*{-0.05in}

The codes in the first subsection below were written and ran in MATLAB version R2021b, while for the second subsection we used Python~3.8. The tests were ran on a desktop of Intel Core i7-4770 CPU 3.40GHz with 32GB RAM, under Windows 10 (64-bit).\vspace*{-0.2in}

\subsection{Smooth DC Models in Biochemistry}\vspace*{-0.1in}

Here we consider the problem motivating the development of BDCA in \cite{AragonArtacho2018}, which consists of finding a steady state of a dynamical equation arising in the modeling of {\em biochemical reaction networks}. We ran our experiments on the same 14 biochemical reaction network models tested in \cite{AragonArtacho2018,MR4078808}.
The problem can be modeled as finding a zero of the function
$$f(x):=\left([F,R]-[R,F]\right)\exp\left(w+[F,R]^\tr x\right),$$
where $F,R\in\mathbb{Z}_{\geq 0}^{m\times n}$ denote the forward and reverse \emph{stoichiometric matrices}, respectively, where $w\in\mathbb{R}^{2n}$ is the componentwise logarithm of the \emph{kinetic parameters}, where $\exp(\cdot)$ is the componentwise exponential function, and where $[\,\cdot\,,\cdot\,]$ stands for the horizontal concatenation operator. Finding a zero of $f$ is equivalent to minimizing the function $\varphi(x):=\|f(x)\|^2$, which can be expressed as a {\em difference of the convex functions}
\begin{equation}\label{eq:bio_DCA}
g(x):=2\left(\|p(x)\|^2+\|c(x)\|^2\right)\quad\text{and}\quad h(x):=\|p(x)+c(x)\|^2,
\end{equation}
where the functions $p(x)$ and $c(x)$ are given by
\begin{equation*}
p(x):=[F,R]\exp\left(w+[F,R]^\tr x\right)\quad\text{and}\quad c(x):=[R,F]\exp\left(w+[F,R]^\tr x\right).
\end{equation*}
In addition, it is also possible to write
\begin{equation*}
\varphi(x)=\|f(x)\|^2=\|p(x)-c(x)\|^2=\|p(x)\|^2+\|c(x)\|^2-2p(x)c(x),
\end{equation*}
and so $\varphi(x)$ can be decomposed as the difference of the functions
\begin{equation}\label{eq:bio_ours}
g(x):=\|p(x)\|^2+\|c(x)\|^2 \quad\text{and}\quad h(x)=2p(x)c(x)
\end{equation}
with $g$ being convex. Therefore, $\nabla^2 g(x)$ is $0$-lower definite, and minimizing $\varphi$ can be tackled with Algorithm~\ref{alg:1} by choosing $\rho_k\geq\zeta$ for some fixed $\zeta>0$.
As shown in \cite{AragonArtacho2018}, the function $\varphi$ is real analytic and thus satisfies the Kurdyka--{\L}ojasiewicz assumption of Theorem~\ref{COR:Kur-Loj}, but as observed in \cite[Remark~5]{AragonArtacho2018}, a linear convergence rate cannot be guaranteed.

Our first task in the conducted experiments was to decide how to set the parameters $\zeta$ and $\rho_k$. We compared the strategy of taking $\rho_k$ equal to some fixed value for all $k$, setting a decreasing sequence bounded from below by $\zeta$, and choosing $\rho_k=c\|w_k\|+\zeta$ for some constant $c>0$. In spite of Remark~\ref{rem:theorem}(ii), $\zeta$ was added in the last strategy to guarantee both Theorem~\ref{The01}(ii) and Theorem~\ref{COR:Kur-Loj}. We took $\zeta=10^{-8}$ and a constant $c=5$, which worked well in all the models. We tried several options for the decreasing strategy, of which a good choice seemed to be $\rho_k=\frac{\|w_0\|}{10^{\lfloor k/50\rfloor}}+\zeta$, where $\lfloor\cdot\rfloor$ denotes the floor function (i.e., the parameter was initially set to $\|w_0\|$ and then divided by $10$ every 50 iterations). The best option was this decreasing strategy, as can be observed in the two models in Figure~\ref{fig:bio_rhos}, and this was the choice for our subsequent tests.\vspace*{-0.1in}

\begin{figure}[ht!]
\centering
\includegraphics[width=.49\textwidth]{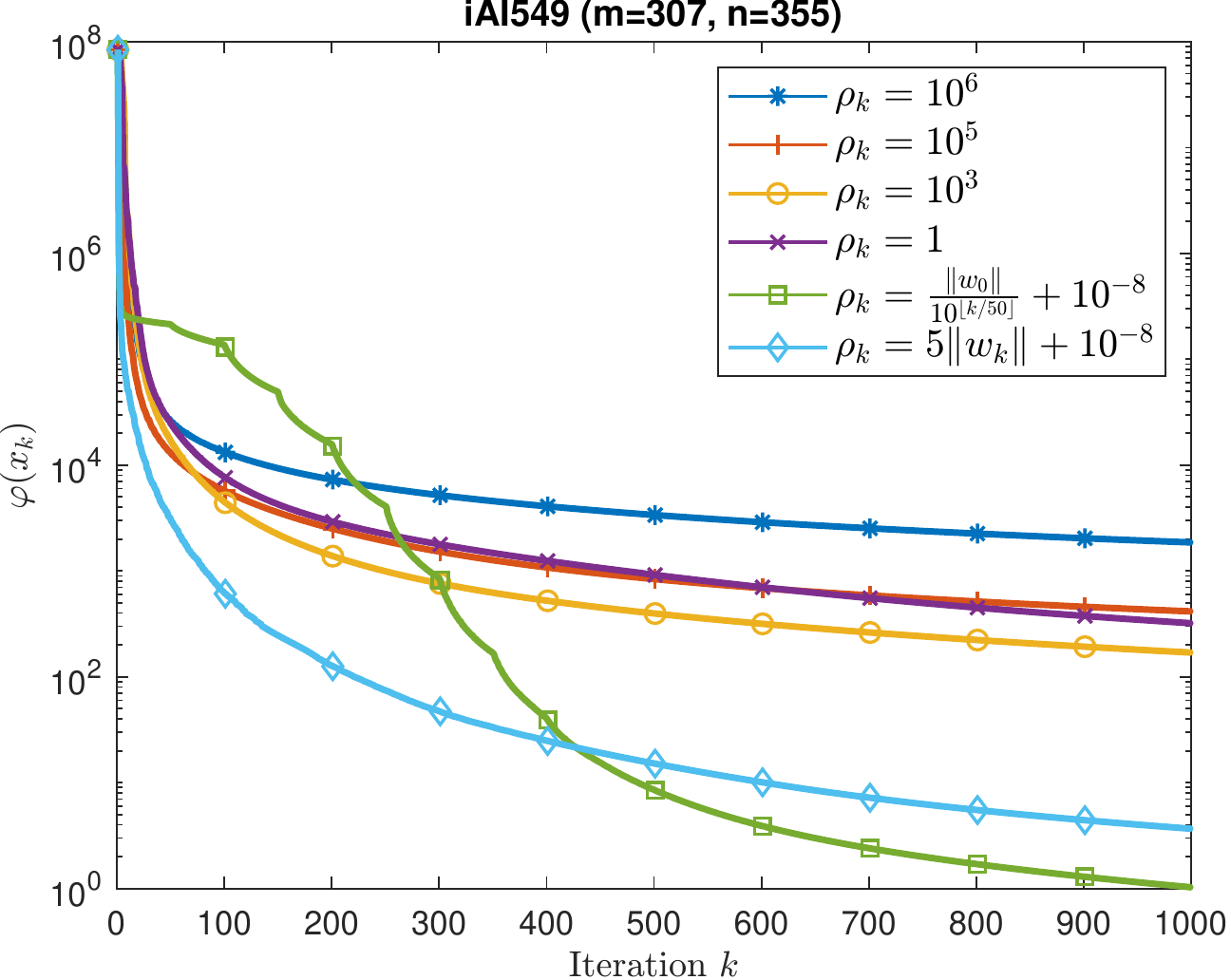} 
\includegraphics[width=.49\textwidth]{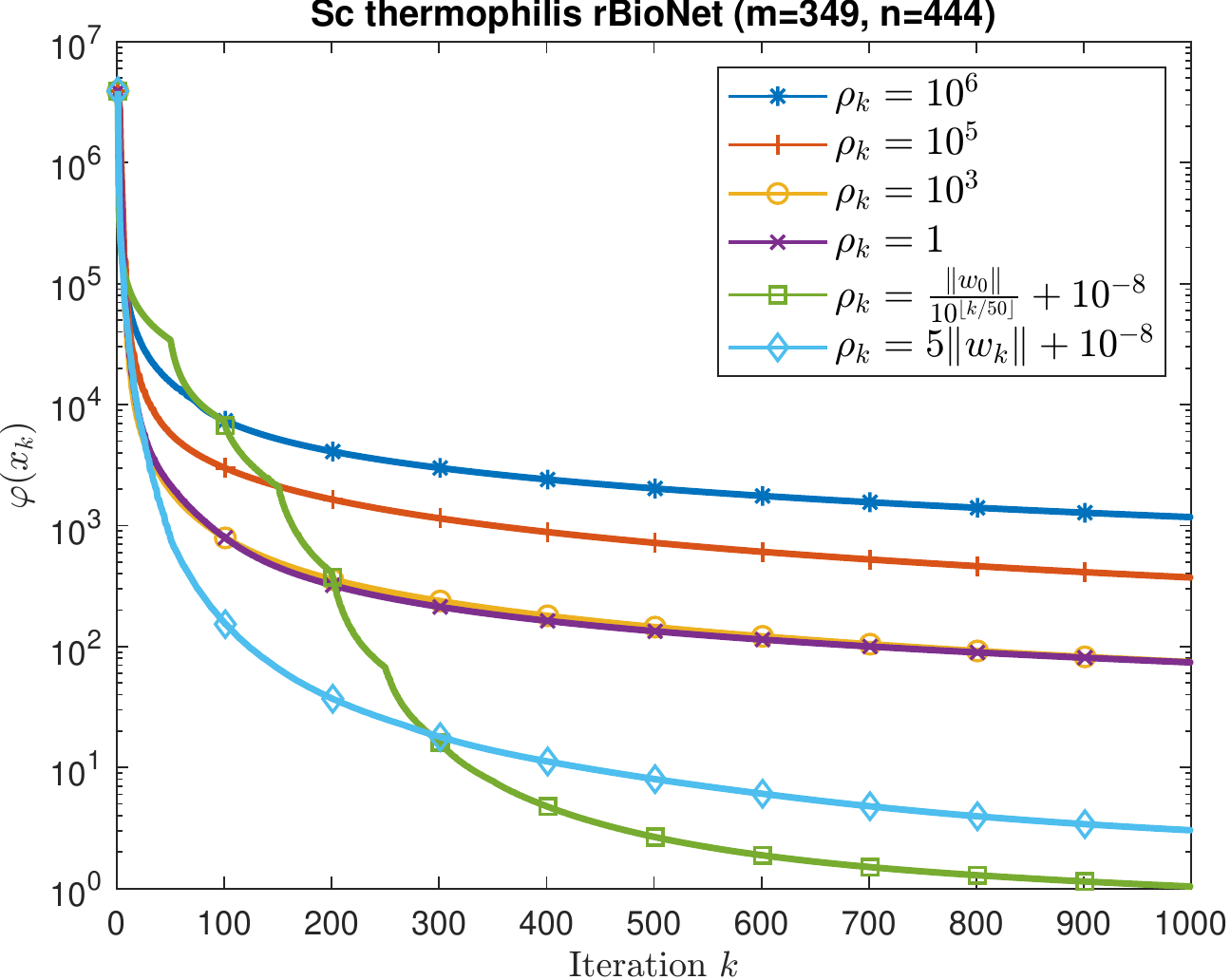}
\caption{Comparison of the objective values for three strategies for setting the regularization parameter $\rho_k$: constant (with values $10^6$, $10^5$, $10^3$ and $1$), decreasing,  and adaptive with respect to the value of $\|w_k\|$.}\label{fig:bio_rhos}
\end{figure}
\vspace*{-0.3in}

\begin{experiment}\label{exp1} For finding a steady state of each of the 14 biochemical models, we compared the performance of Algorithm~\ref{alg:1} and BDCA with self-adaptive strategy, which was the fastest method tested in~\cite{MR4078808} (on average, 6.7 times faster than DCA). For each model, 5 kinetic parameters were randomly chosen with coordinates uniformly distributed in $(-1,1)$, and 5 random starting points with random coordinates in $(-2,2)$ were picked. BDCA was ran using the same parameters as in~\cite{MR4078808}, while we took $\sigma=\beta=0.2$ for Algorithm~\ref{alg:1}.

We considered two strategies for setting the trial stepsize $\overline{\tau}_k$ in Step~4 of Algorithm~\ref{alg:1}: constantly initially set to 50, and self-adaptive strategy (Algorithm~\ref{alg:self-adaptive}) with $\gamma=2$  and $t_{\min}=10^{-8}$. For each model and each random instance, we computed 500 iterations of BDCA with self-adaptive strategy and then ran Algorithm~\ref{alg:1} until the same value of the target function $\varphi$ was reached. As in~\cite{AragonArtacho2018}, the BDCA subproblems were solved by using the function \texttt{fminunc} with \texttt{optimoptions(\char13 fminunc\char13,} \texttt{\char13 Algorithm\char13,} \texttt{\char13 trust-region\char13,} \texttt{\char13 GradObj\char13,} \texttt{\char13 on\char13,} \texttt{\char13 Hessian\char13,} \texttt{\char13 on\char13,} \texttt{\char13 Display\char13,} \texttt{\char13 off\char13,} \texttt{\char13 TolFun\char13,} \texttt{1e-8,} \texttt{\char13 TolX\char13,} \texttt{1e-8)}.

The results are summarized in Figure~\ref{fig:bio_ratio}, where we plot the ratios of the running times between BDCA with self-adaptive stepsize and Algorithm~\ref{alg:1} with constant trial stepsize against Algorithm~\ref{alg:1} with self-adaptive stepsize. On average, Algorithm~\ref{alg:1} with self-adaptive strategy was $6.69$ times faster than BDCA, and was $1.33$ times faster than Algorithm~\ref{alg:1} with constant strategy. The lowest ratio for the times of self-adaptive Algorithm~\ref{alg:1} and BDCA was $3.17$. Algorithm~\ref{alg:1} with self-adaptive stepsize was only once (out of the 70 instances) slightly slower (a ratio of 0.98) than with the constant strategy.\vspace*{-0.05in}

\begin{figure}[ht!]
\centering
\includegraphics[width=.7\textwidth]{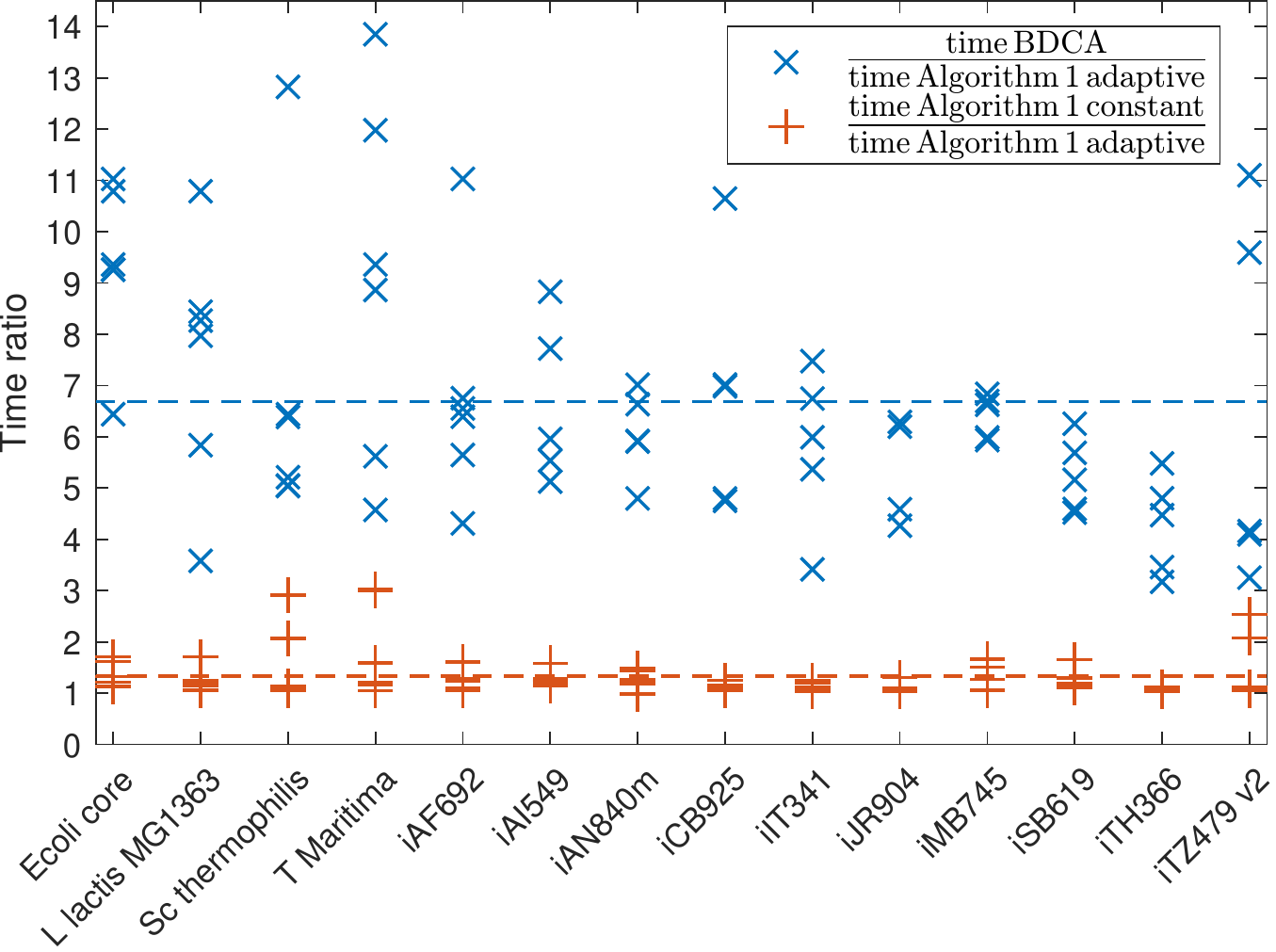}
\caption{Ratios of the running times of Algorithm~\ref{alg:1} with constant stepsize and BDCA with self-adaptive stepsize to Algorithm~\ref{alg:1} with self-adaptive stepsize. For each of the models, the algorithms were run using the same random starting points. The overall average ratio is represented with a dashed line}\label{fig:bio_ratio}
\end{figure}\vspace*{-0.2in}

In Figure~\ref{fig:bio_comparison}, we plot the values of the objective function for each algorithm and also include for comparison the results for DCA and BDCA without self-adaptive strategy. The self-adaptive strategy also accelerates the performance of Algorithm~\ref{alg:1}. We can observe in Figure~\ref{fig:bio_taus} that there is a correspondence between the drops in the objective value and large increases of the stepsizes $\tau_k$ (in a similar way to what was shown for BDCA in~\cite[Fig.~12]{MR4078808}).\vspace*{-0.15in}

\begin{figure}[ht!]
\centering
\includegraphics[width=.49\textwidth]{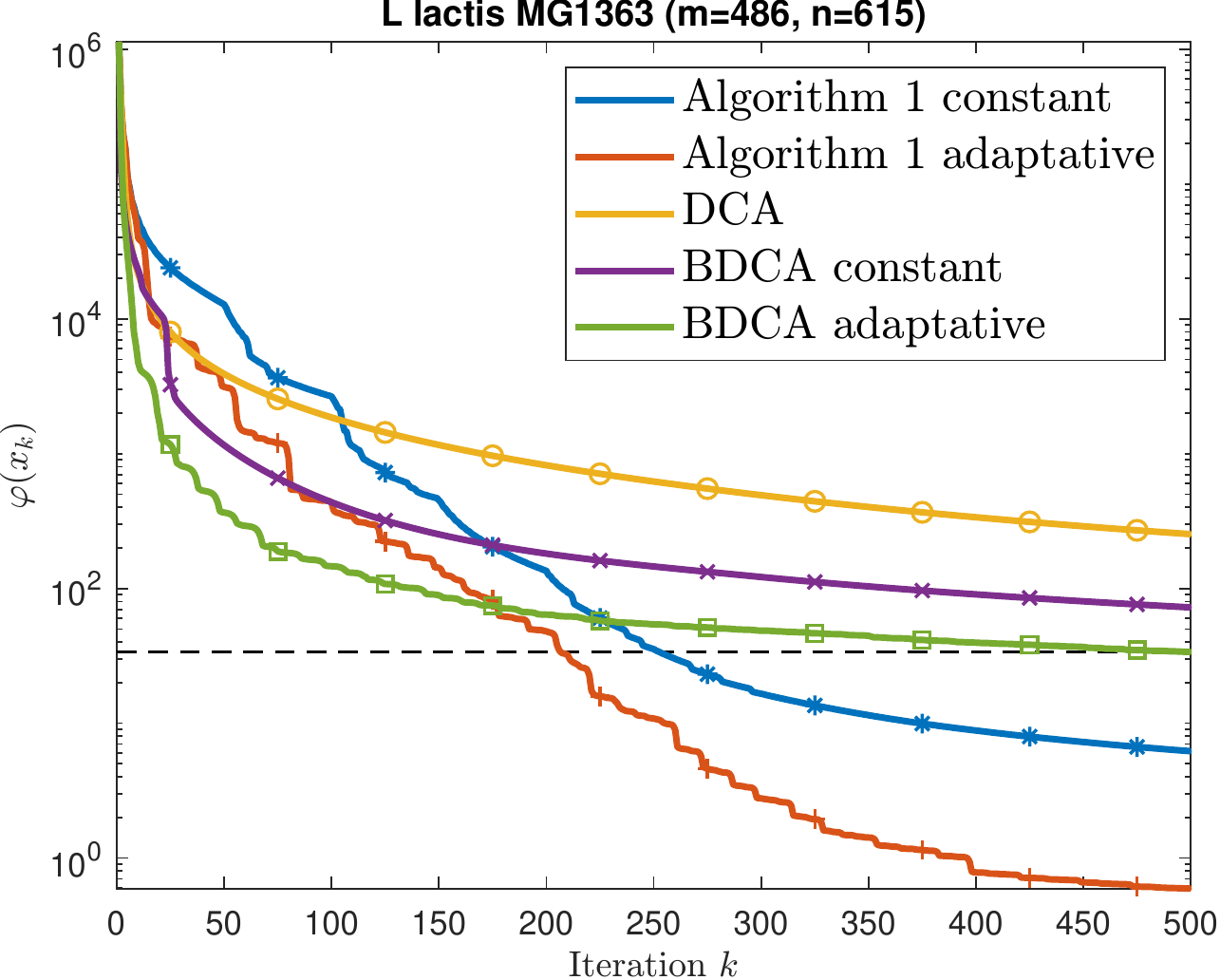} 
\includegraphics[width=.49\textwidth]{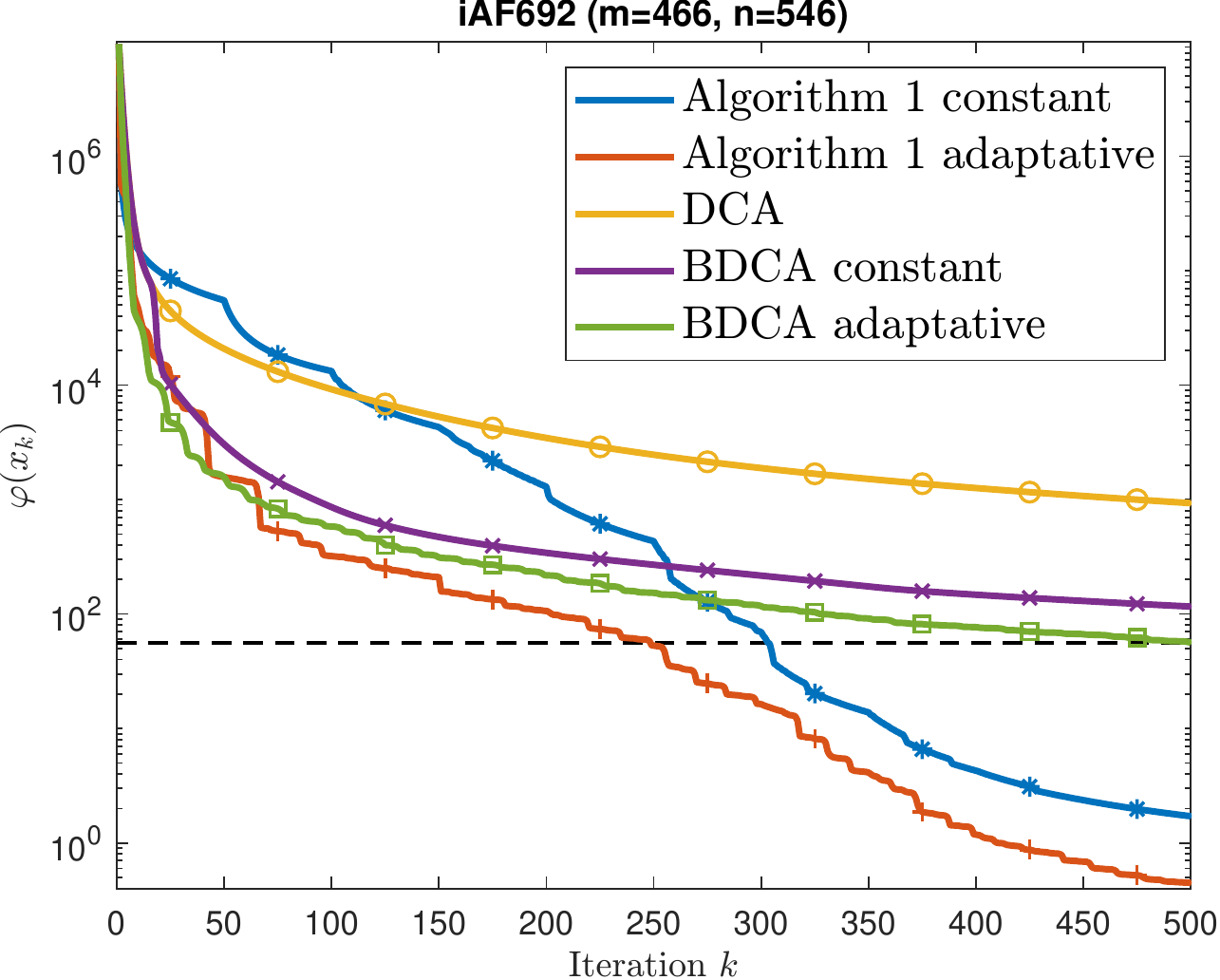}
\caption{Value of the objective function (with logarithmic scale) of Algorithm~\ref{alg:1}, DCA and BDCA for two biochemical models. The value attained after 500 iterations of BDCA with self-adaptive stepsize is shown by a dashed line.}\label{fig:bio_comparison}
\end{figure}

\begin{figure}[ht!]
\centering
\includegraphics[width=.49\textwidth]{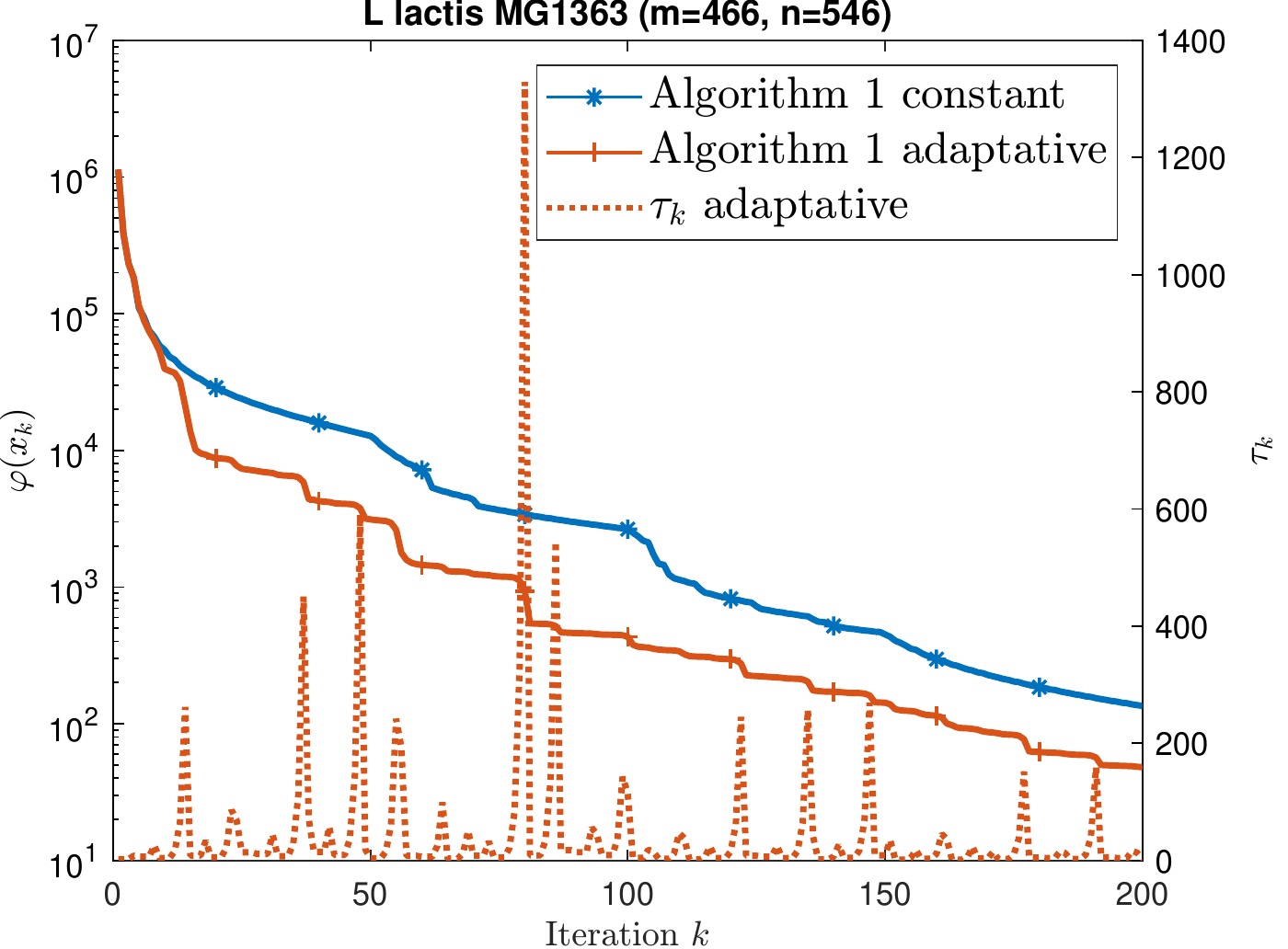} 
\includegraphics[width=.49\textwidth]{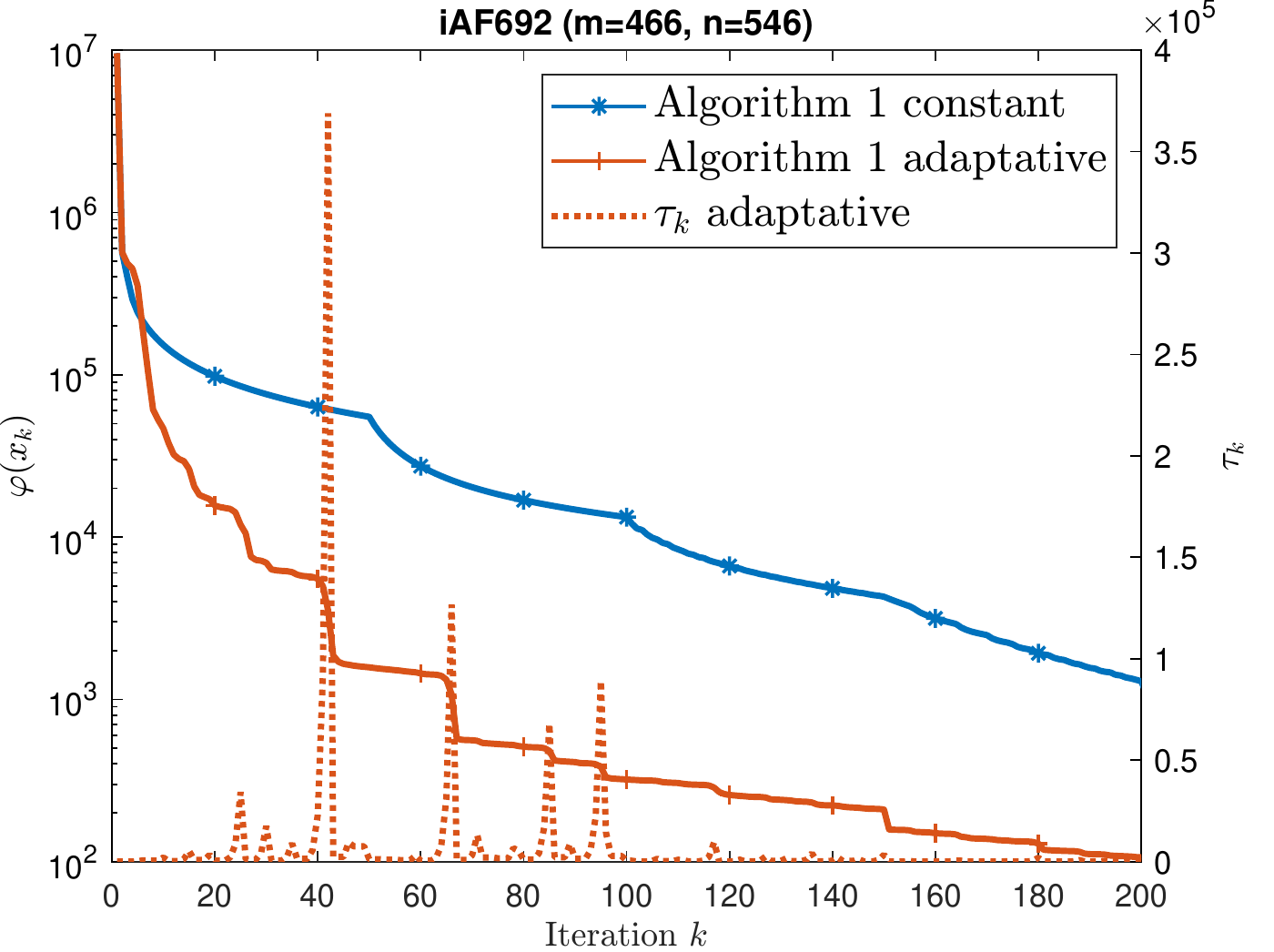}
\caption{Comparison of the self-adaptive and the constant (with $\overline{\tau}_k = 50$) choices for the trial
stepsizes  in Step 4~of Algorithm~\ref{alg:1} for two biochemical models. The plots include two
scales, a logarithmic one for the objective function values and a linear one for the stepsizes (which
are represented with discontinuous lines).}\label{fig:bio_taus}
\end{figure}
\end{experiment}

\subsection{Solving Constrained Quadratic Optimization Models}
\vspace*{-0.1in}

This subsection contains numerical experiments to solve problems of constrained quadratic optimization formalized by
\begin{align}\label{ProblemQUAD1}
\mbox{minimize }\;\frac{1}{2}x^\tr Q x + b^\tr x\;  \text{ subject to }\;x\in C:=\bigcup_{i=1}^p C_i,	
\end{align}
where $Q$ is a symmetric matrix (not necessarily positive-semidefinite), $b \in \mathbb{R}^n$, and $C_1,\ldots,C_p\subseteq \R^n$ are nonempty, closed, and convex sets.

When $C=\mathbb{B}_r(0)$ (i.e., $p=1$), this problem is referred as the {\em trust-region subproblem}. If $Q$ is positive-semidefinite, then \eqref{ProblemQUAD1} is a problem of {\em convex quadratic programming}. Even when $Q$ is not positive-semidefinite, Tao and An~\cite{Tao1998} showed that this particular instance of problem \eqref{ProblemQUAD1} could be efficiently addressed with the DCA algorithm by using the following DC decomposition:
\begin{equation}\label{eq:Quad_DCA_g_h}
g(x):=\frac{1}{2}\rho\|x\|^2+b^\tr x+\delta_{\mathbb{B}_r(0)},\quad h(x):=\frac{1}{2}x^\tr(\rho I-Q)x,
\end{equation}
where $\rho\geq\|Q\|_2$. However, this type of decomposition would not be suitable for problem \eqref{ProblemQUAD1} when $C$ is not convex.

As shown in Subsection~\ref{subsec:5.1}, problem \eqref{ProblemQUAD1} for $p\geq 1$ can be reformulated by using  FBE \eqref{FBE} to be tackled with Algorithm~\ref{alg:3} with $\lambda\in(0,\frac{1}{\|Q\|_2})$. Although the decomposition in \eqref{func_h_Example55} may not be suitable for DCA when $Q$ is not positive-definite, it can be regularized by adding $\frac{1}{2}\rho\|x\|^2$ to both $g$ and $h$ with $\rho\geq\max\{0,-2\lambda_{\min}(Q)\}$. Such a regularization would guarantee the convexity of the resulting functions $g$ and $h$ given by
\begin{align}
 g(x)& :=  \frac{1}{2}x^\tr \left(  Q + (\rho+\lambda^{-1}) I\right) x + b^\tr x,  \label{eq:Quad_FBE_g_h}\\
 h(x)& :=  \frac{1}{2}x^\tr  \left(2Q+\rho I\right) x + b^\tr x+ \Asp{\lambda}{\delta_C}((I- \lambda Q)x-\lambda b). \label{eq:Quad_FBE_g_h_bis}
 \end{align}
The function $g$ in \eqref{eq:Quad_DCA_g_h} is not smooth, but the function $g$ in~\eqref{eq:Quad_FBE_g_h} is. Then it is possible to apply BDCA \cite{MR4078808} to formulation \eqref{eq:Quad_FBE_g_h}--\eqref{eq:Quad_FBE_g_h_bis} in order to accelerate the convergence of DCA. Note that it would also be possible to do it with \eqref{eq:Quad_DCA_g_h} if the $\ell_1$ or $\ell_{\infty}$ balls were used; see~\cite{Artacho2019} for more details.\vspace*{0.03in}

Let us describe two numerical experiments to solve problem \eqref{ProblemQUAD1}. \vspace*{-0.05in}

\begin{experiment}\label{exp2}
Consider \eqref{ProblemQUAD1} with $C=\mathbb{B}_r(0)$ and replicate the hardest setting in \cite{Tao1998}, which was originally considered in~\cite{More1983}. Specifically, in this experiment we generated potentially difficult cases by setting $Q:=UDU^\tr$ for some diagonal matrix $D$ and orthogonal matrix $U:=U_1U_2U_3$ with $U_j:=I-2u_ju_j^\tr/\|u_j\|^2$, $j=1,2,3$. The components of $u_j$ were random numbers uniformly distributed in $(-1,1)$, while the elements in the diagonal of $D$ were random numbers in $(-5,5)$. We took $b:=Uz$ for some vector $z$ whose elements were random numbers uniformly distributed in $(-1,1)$ except for the component corresponding to the smallest element of $D$, which was set to $0$. The radius $r$ was randomly chosen in the interval $(\|d\|,2\|d\|)$, where $d_i:=z_i/(D_{ii}-\lambda_{\min}(D))$ if $D_{ii}\neq \lambda_{\min(D)}$ and $0$ otherwise.

For each $n\in\{100,200,\ldots,900,1000,1250,1500,\ldots,3750,4000\}$, we generated 10 random instances, took for each instance a random starting point in $\mathbb{B}_r(0)$, and ran from it the four algorithms described above: DCA applied to formulation~\eqref{eq:Quad_DCA_g_h} (without FBE), DCA and BDCA applied to~\eqref{eq:Quad_FBE_g_h}--\eqref{eq:Quad_FBE_g_h_bis}, and Algorithm~\ref{alg:3}. We took $\lambda=0.8/\|Q\|_2$ as the parameter for FBE (both for DCA and Algorithm~\ref{alg:3}). The regularization parameter $\rho$ was chosen as $\max\{0,-2\lambda_{\min}(Q)\}$ for DCA with FBE and $0.1+\max\{0,-2\lambda_{\min}(Q)\}$ for BDCA, as $h$ should be strongly convex. Both Algorithm~\ref{alg:3} and BDCA were ran with the self-adaptive trial stepsize for the backtracking step introduced in~\cite{MR4078808} with parameters $\sigma=\beta=0.2$ and $\gamma=4$, and with $t_{\min}=10^{-6}$. For the shake of fairness, we did not compute function values for the runs of DCA at each iteration, since it is not required by the algorithm. Instead, we used for both versions of DCA the stopping criterion from~\cite{Tao1998} that $er\leq 10^{-4}$, where
$$er=\left\{\begin{array}{lc}
\left\|x^{k+1}-x^k\right\| /\left\|x^k\right\| & \text { if }\left\|x^k\right\|>1, \\
\left\|x^{k+1}-x^k\right\|& \text { otherwise.}
\end{array}\right.$$
As DCA with FBE was clearly the slowest method, we took the function value of the solution returned by DCA without FBE as the target value for both Algorithm~\ref{alg:3} and BDCA, so these algorithms were stopped when that function value was reached. In Figure~\ref{fig:trust_region_ratio_small}, we plot the time ratio of each algorithm against Algorithm~\ref{alg:3}. On average, Algorithm~\ref{alg:3} was more than 5 times faster than DCA with FBE and more than 2 times faster than DCA without FBE. BDCA greatly accelerated the performance of DCA with FBE, but still Algorithm~\ref{alg:3} was more than 1.5 times faster. Only for size 300, the performance of DCA without FBE was comparable to that of  Algorithm~\ref{alg:3}. We observe on the right plot that the advantage of Algorithm~\ref{alg:3} is maintained for larger sizes.

\begin{figure}[ht!]
\centering
\includegraphics[width=.49\textwidth]{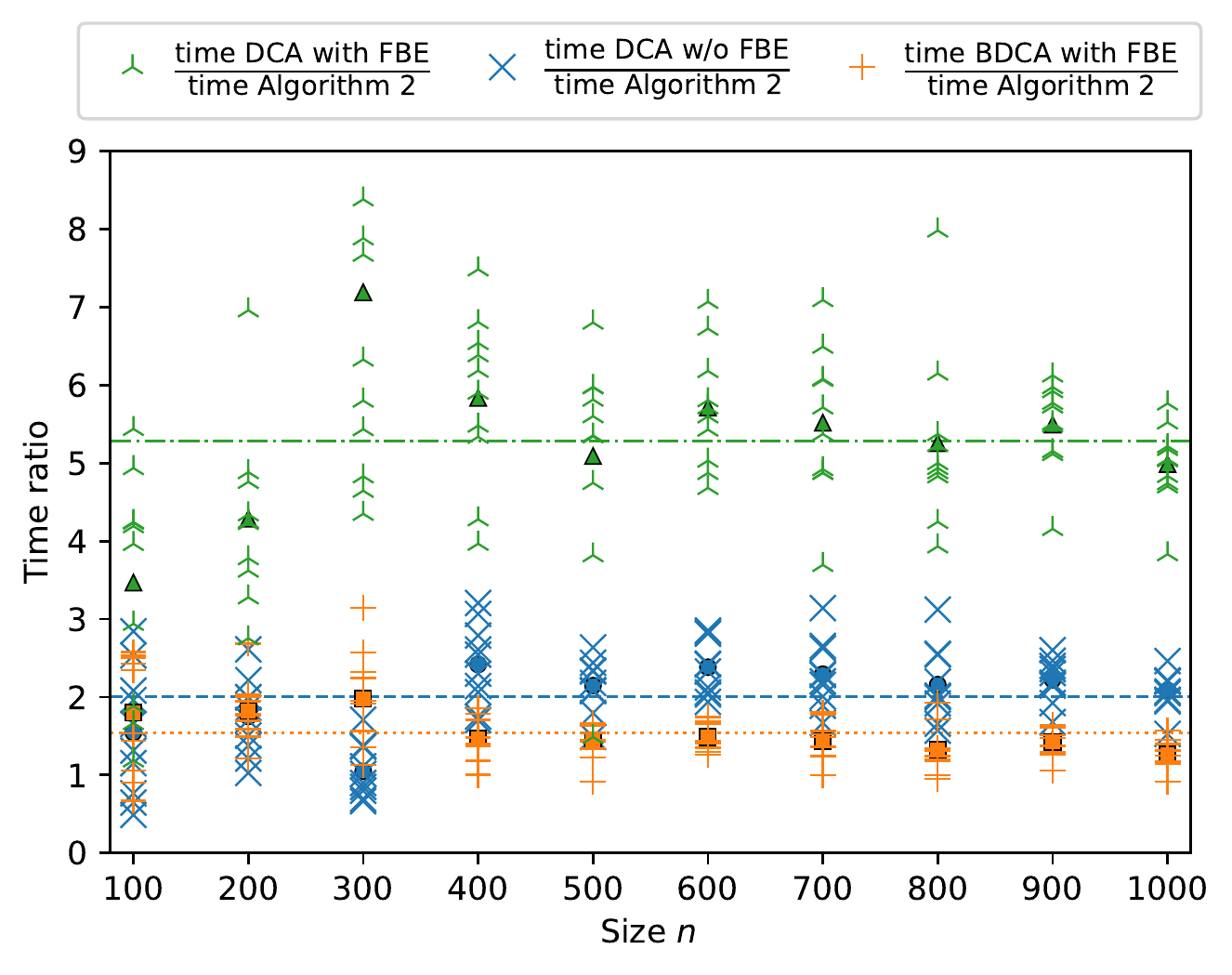} 
\includegraphics[width=.49\textwidth]{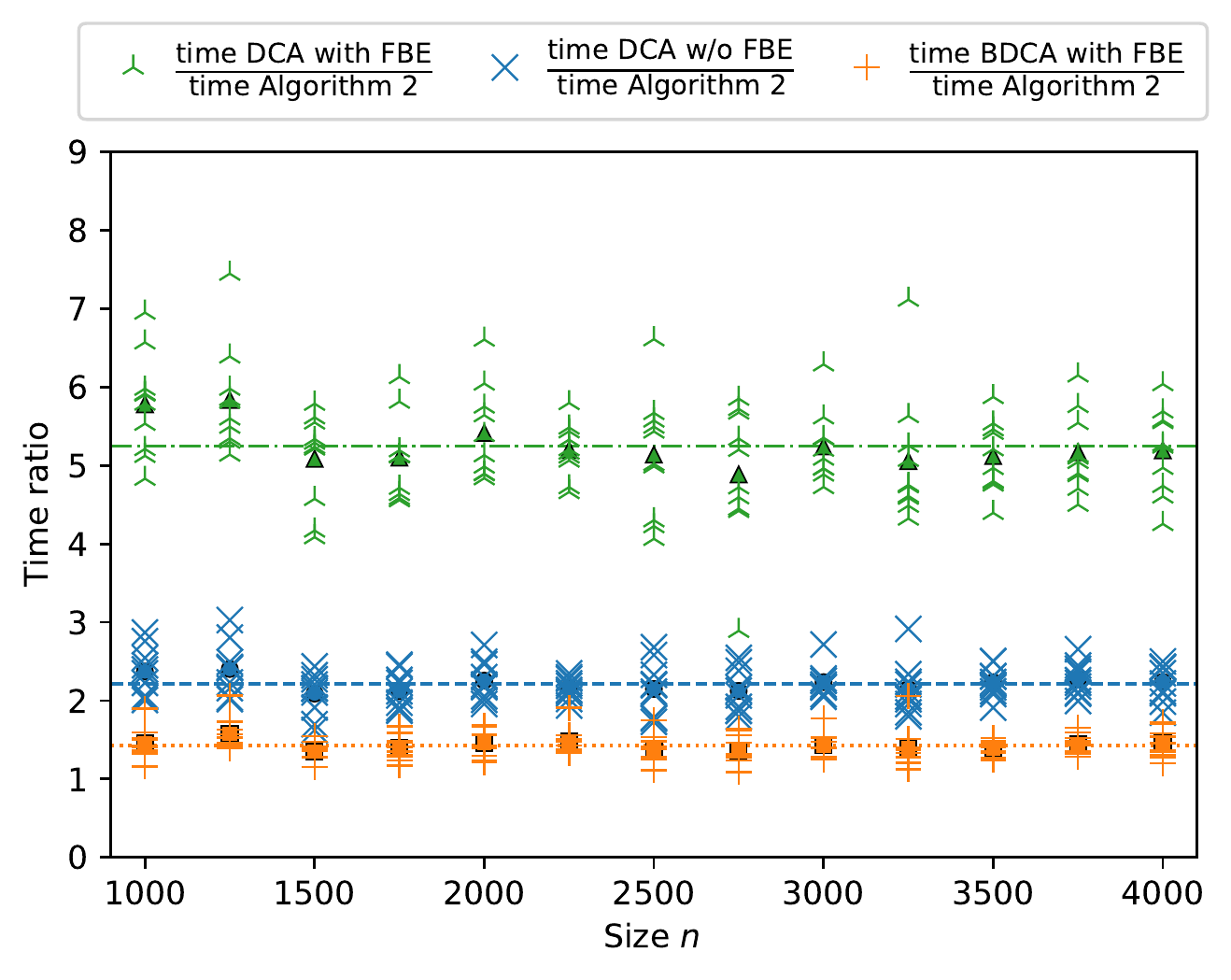}
\caption{Time ratio for 10 random instances of DCA with FBE, DCA without FBE, and BDCA with respect to Algorithm~\ref{alg:3}. Average ratio within each size is represented with a triangle for DCA with FBE, with a square for DCA without FBE and with a circle for BDCA. The overall average ratio for each pair of algorithms is represented by a dotted line.}\label{fig:trust_region_ratio_small}
\end{figure}
\end{experiment}

\begin{experiment} With the aim of finding the minimum of a quadratic function with integer and box constraints, we modified the setting of Experiment~2 and considered instead a set $C$ composed by $9^n$ balls of various radii centered at $\{-4,-3,-2,-1,0,1,2,3,4\}^n$, with $n\in\{2,10,25,50,100,200,500,1000\}$. As balls of radius $1/2\sqrt{n}$ cover the region $[-4,4]^n$, we ran our tests with balls of radii $c/2\sqrt{n}$ with $c\in\{0.1,0.2,\ldots,0.8, 0.9\}$. This time we considered {\em both convex and nonconvex} objective functions. The nonconvex case was generated as in Experiment~2, while for the convex case, the elements of the diagonal of $D$ were chosen as random numbers uniformly distributed in $(0,5)$. 

For each $n$ and~$r$, 100 random instances were generated. For each instance, a starting point was chosen with random coordinates uniformly distributed in $[-5,5]^n$. As the constraint sets are nonconvex, FBE was also needed to run DCA. The results are summarized in Table~\ref{tbl:integer}, where for each $n$ and each radius, we counted the number of instances (out of 100) in which the value of $\varphi_\lambda$ at the rounded output of DCA and BDCA was lower and higher than that of Algorithm~\ref{alg:3} when ran from the same starting point. We used the same parameter settings for the algorithms as in Experiment~2. Finally, we plot in Figure~\ref{fig:integer} two instances in $\mathbb{R}^2$ in which Algorithm~\ref{alg:3} reached a better solution.\vspace*{-0.25in}

\begin{table}[ht!]
\begin{subtable}[t]{\textwidth}\centering
\scalebox{.85}{\centering
\begin{tabular}{c c cc cc cc cc c}
\toprule
\multicolumn{2}{c}{}&\multicolumn{9}{c}{Radius of the balls}\\
\cmidrule[.7pt]{3-11}
& Alg.~\ref{alg:3} vs & $\frac{1}{20}\sqrt{n}$ & $\frac{2}{20}\sqrt{n}$ & $\frac{3}{20}\sqrt{n}$ & $\frac{4}{20}\sqrt{n}$ & $\frac{5}{20}\sqrt{n}$ & $\frac{6}{20}\sqrt{n}$ & $\frac{7}{20}\sqrt{n}$ & $\frac{8}{20}\sqrt{n}$ & $\frac{9}{20}\sqrt{n}$ \\
\midrule[.7pt]
\multirow{ 2}{*}{$n=2$}& DCA & 0/34 & 1/20 & 1/26 & 1/19 & 0/19 & 0/18 & 0/3 & 1/1 & 0/2\\
& BDCA & 6/12 & 2/13 & 3/14 & 4/9 & 0/4 & 1/10 & 0/2 & 1/1 & 0/2\\
\midrule[.5pt]
\multirow{ 2}{*}{$n=10$}& DCA & 2/89 & 2/83 & 1/66 & 5/53 & 8/28 & 3/7 & 1/1 & 3/1 & 1/0\\
& BDCA & 21/68 & 38/53 & 33/39 & 23/24 & 18/16 & 2/7 & 0/0 & 0/1 & 0/0 \\
\midrule[.5pt]
\multirow{ 2}{*}{$n=25$}& DCA & 0/99 & 0/98 & 2/87 & 11/58 & 9/32 & 3/8 & 2/9 & 2/2 & 5/4\\
& BDCA & 16/83 & 29/71 & 40/58 & 37/40 & 13/26 & 2/3 & 0/1 & 0/0 & 1/2\\
\midrule[.5pt]
\multirow{ 2}{*}{$n=50$}& DCA & 0/100 & 0/100 & 0/91 & 2/86 & 13/41 & 14/12 & 9/12 & 6/10 & 12/12\\
& BDCA & 8/92 & 6/94 & 31/69 & 36/53 & 16/28 & 8/8 & 6/5 & 3/4 & 5/3\\
\midrule[.5pt]
\multirow{ 2}{*}{$n=100$}& DCA & 0/100 & 0/100 & 0/99 & 9/87 & 18/49 & 18/31 & 12/22 & 18/20 & 11/21\\
& BDCA & 2/98 & 6/94 & 39/61 & 36/61 & 23/33 & 16/14 & 9/8 & 9/8 & 13/9\\
\midrule[.5pt]
\multirow{ 2}{*}{$n=200$}& DCA & 0/100 & 0/100 & 0/100 & 1/98 & 23/64 & 31/41 & 25/29 & 22/30 & 20/41\\
& BDCA & 3/97 & 2/98 & 38/62 & 37/63 & 33/39 & 27/17 & 18/18 & 14/13 & 16/18\\
\midrule[.5pt]
\multirow{ 2}{*}{$n=500$}& DCA & 0/100 & 0/100 & 0/100 & 1/99 & 6/94 & 15/80 & 27/61 & 29/65 & 36/48\\
& BDCA & 0/100 & 1/99 & 41/59 & 44/56 & 33/63 & 25/56 & 34/39 & 32/47 & 17/35\\
\bottomrule
\end{tabular}}
\caption{Convex case}
\end{subtable}

\begin{subtable}[h]{\textwidth}\centering
\scalebox{.85}{\centering
\begin{tabular}{c c cc cc cc cc c}
\toprule
\multicolumn{2}{c}{}&\multicolumn{9}{c}{Radius of the balls}\\
\cmidrule[.7pt]{3-11}
& Alg.~\ref{alg:3} vs & $\frac{1}{20}\sqrt{n}$ & $\frac{2}{20}\sqrt{n}$ & $\frac{3}{20}\sqrt{n}$ & $\frac{4}{20}\sqrt{n}$ & $\frac{5}{20}\sqrt{n}$ & $\frac{6}{20}\sqrt{n}$ & $\frac{7}{20}\sqrt{n}$ & $\frac{8}{20}\sqrt{n}$ & $\frac{9}{20}\sqrt{n}$ \\
\midrule[.7pt]
\multirow{ 2}{*}{$n=2$}& DCA & 1/8 & 1/9 & 1/10 & 1/6 & 0/6 & 0/6 & 0/8 & 3/2 & 1/0\\
& BDCA & 2/4 & 1/4 & 1/3 & 3/4 & 0/5 & 0/4 & 0/6 & 3/2 & 1/0\\
\midrule[.5pt]
\multirow{ 2}{*}{$n=10$}& DCA & 9/39 & 4/39 & 7/39 & 4/35 & 10/30 & 3/27 & 5/45 & 2/34 & 8/29\\
& BDCA & 9/31 & 11/33 & 13/29 & 6/31 & 11/29 & 6/25 & 7/38 & 5/29 & 10/30\\
\midrule[.5pt]
\multirow{ 2}{*}{$n=25$}& DCA & 6/69 & 13/67 & 7/62 & 5/61 & 10/53 & 3/59 & 6/56 & 3/72 & 3/66\\
& BDCA & 16/58 & 16/63 & 16/55 & 12/48 & 9/52 & 11/52 & 13/52 & 12/57 & 11/58\\
\midrule[.5pt]
\multirow{ 2}{*}{$n=50$}& DCA & 11/81 & 10/79 & 8/87 & 5/90 & 3/87 & 4/80 & 2/86 & 5/89 & 8/81\\
& BDCA & 24/68 & 21/64 & 23/70 & 17/73 & 14/75 & 9/73 & 10/75 & 18/74 & 16/71\\
\midrule[.5pt]
\multirow{ 2}{*}{$n=100$}& DCA & 4/96 & 6/94 & 4/94 & 5/94 & 4/96 & 3/97 & 2/98 & 7/91 & 9/91\\
& BDCA & 15/85 & 16/83 & 18/80 & 14/84 & 17/83 & 11/89 & 9/91 & 20/79 & 19/80\\
\midrule[.5pt]
\multirow{ 2}{*}{$n=200$}& DCA & 4/96 & 4/96 & 4/96 & 2/98 & 1/99 & 2/98 & 4/96 & 3/97 & 0/100\\
& BDCA & 11/89 & 16/84 & 11/89 & 8/92 & 6/94 & 11/89 & 10/90 & 13/87 & 8/92\\
\midrule[.5pt]
\multirow{ 2}{*}{$n=500$}& DCA & 1/99 & 2/98 & 0/100 & 0/100 & 0/100 & 1/99 & 1/99 & 2/98 & 1/99\\
& BDCA & 12/88 & 17/83 & 15/85 & 9/91 & 15/85 & 11/89 & 9/91 & 18/82 & 20/80\\
\bottomrule
\end{tabular}}
\caption{Nonconvex case}
\end{subtable}
\caption{For different values of  $n$ (space dimension) we computed 100 random instances of problem~\eqref{ProblemQUAD1} with $Q$ positive definite and $C$ formed by the union of balls whose centers have integer coordinates between $-4$ and $4$. We counted the number of instances in which DCA and BDCA obtained a lower/upper value than Algorithm~\ref{alg:3}.}\label{tbl:integer}
\end{table}\vspace*{-0.15in}

\begin{figure}[ht!]
\centering
\includegraphics[height=.38\textwidth]{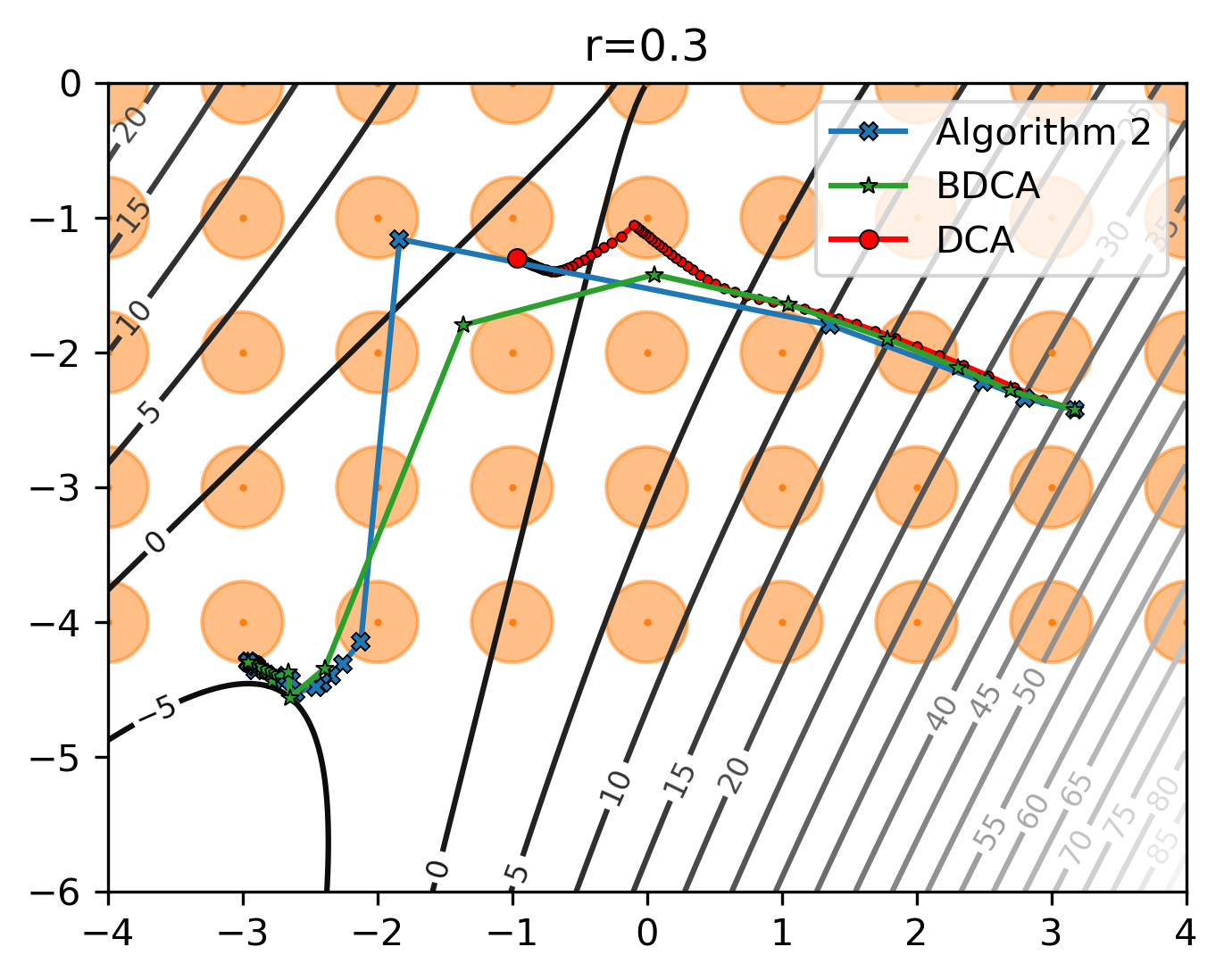}
\includegraphics[height=.38\textwidth]{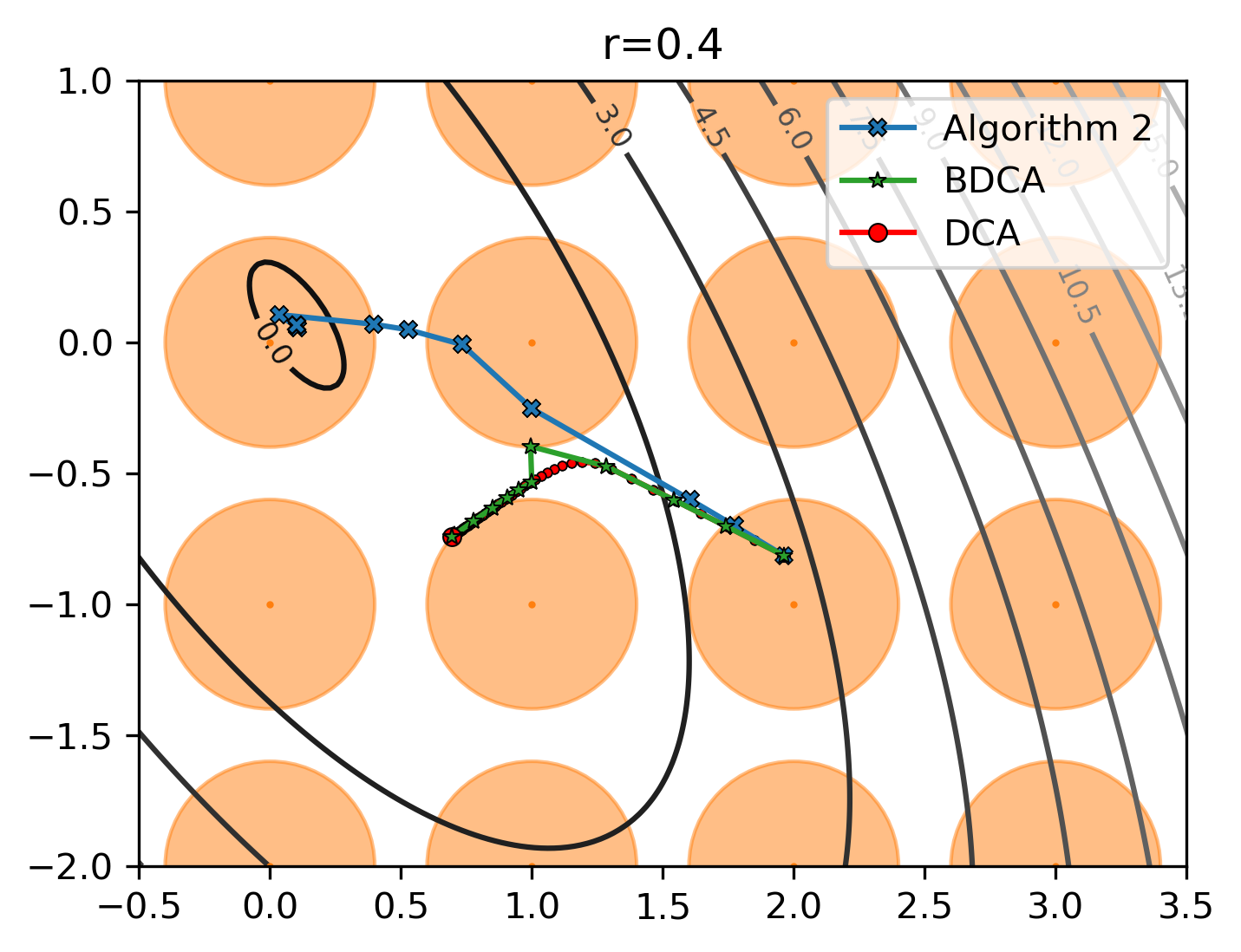}
\caption{Two instances of problem~\eqref{ProblemQUAD1}. On the left, both line searches of Algorithm~\ref{alg:3} and BDCA help to reach a better solution for a nonconvex quadratic function,
while only Algorithm~\ref{alg:3} succeeds on the right for the convex case.}\label{fig:integer} 
\end{figure}
\end{experiment}\vspace*{0.15in}

\section{Conclusion and Future Research}\label{sec:7}\vspace*{-0.05in}

This paper proposes and develops a novel RCSN method to solve problems of difference programming whose objectives are represented as differences of generally nonconvex functions. We establish well-posedness of the proposed algorithm and its global convergence under appropriate assumptions. The obtained results exhibit advantages of our algorithm over known algorithms for DC programming when both functions in the difference representations are convex. We also develop specifications of the main algorithm in the case of structured problems of constrained optimization and conduct numerical experiments to confirm the efficiency of our algorithms in solving practical models. 

In the future research, we plan to relax assumptions on the program data ensuring the linear, superlinear, and quadratic convergence rates for RCSN and also extend the spectrum of applications to particular classes of constrained optimization problems as well as to practical modeling. 

\end{document}